\documentclass[11pt,a4paper]{article}

\usepackage{amsmath}
\usepackage{amsfonts}
\usepackage{amssymb,color}

\usepackage{amsthm}
\usepackage{graphicx}
\usepackage[T1]{fontenc}     
\usepackage{color,graphicx}
\usepackage{epstopdf}
\usepackage{graphics}
\usepackage{subfigure}
\usepackage{psfrag}
\graphicspath{{./Figures/}}
\usepackage{url}
\usepackage{hyperref}

\usepackage[french,english]{babel}
\usepackage[T1]{fontenc}
\usepackage[utf8]{inputenc}

\theoremstyle{plain}

\newtheorem{theorem}{Theorem}

\newtheorem{defi}{Definition}

\newtheorem{lemma}{Lemma}
\newtheorem{prop}{Proposition}
\newtheorem{remark}{Remark}
\newtheorem{assumption}{Assumption}
\newtheorem{notations}{Notations}

\AtBeginDocument{

}

\definecolor{darkblue}{rgb}{0.3,0.3,0.7}
\definecolor{dred}{rgb}{0.7,0.3,0.3}

\def\E{\mathbb{E}}
\def\P{\mathbb{P}}

\def\real{\mathbb{R}}

\def\D{\mathbb{D}}

\def\sumj{\sum_{j=1}^d}
\def\Dp{D^{\mathcal{P}}}
\def\ti{t_{i}}
\def\tio{t_{i+1}}

\title{The It\^o-Tanaka Trick : a non-semimartingale approach}

\author{Laure Coutin\footnote{IMT UMR CNRS 5219, Universit\'e Paul Sabatier 118, route de Narbonne, 31062 Toulouse Cedex 4, France. \; Email: \texttt{laure.coutin@math.univ-toulouse.fr}} \and Romain Duboscq\footnote{INSA de Toulouse, IMT UMR CNRS 5219, Universit\'e de Toulouse, 135 avenue de Rangueil 31077 Toulouse Cedex 4 France. \; Email: \texttt{duboscq@insa-toulouse.fr}} \and Anthony R\'eveillac\footnote{INSA de Toulouse, IMT UMR CNRS 5219, Universit\'e de Toulouse, 135 avenue de Rangueil 31077 Toulouse Cedex 4 France. \; Email: \texttt{anthony.reveillac@insa-toulouse.fr}}}

\begin{document}

\allowdisplaybreaks

\maketitle

\renewcommand{\thefootnote}{\fnsymbol{footnote}}

\begin{abstract}
In this paper we provide an It\^o-Tanaka-Wentzell trick in a non semimartingale context. We apply this result to the study of a fractional SDE with irregular drift coefficient. 
\end{abstract}

\noindent
\textbf{Keywords : }\\
\noindent
\textit{AMS Mathematics Subject Classification 2010: 60H07, 60H10 (Primary); 60G22, 35A02 (Secondary)}

\section{Introduction}
Consider the following ODE : 
\begin{equation}
\label{eq:ODE}
x_t = x_0 + \int_0^t b(s,x_s) ds, \quad t \in [0,T], \quad x_0 \in \real^d,
\end{equation}
with $b:[0,T] \times \real^d \to \real^d$ a given vector field. The well-posedness of this equation is obviously related to the smoothness of the coefficient $b$ and in particular famous counter-examples to uniqueness can be derived even in dimension one. The so-called Peano example fits into that paradigm and consists of choosing : 
$$ d=1, \quad x_0=0, \quad b(t,x):=\sqrt{2} \, \textrm{sgn}(x) \sqrt{|x|},$$
for which any mapping of the form $t \mapsto \pm (t-t_0)^2$ (with $t_0$ in $[0,T]$) is solution to (\ref{eq:ODE}). However, the seminal works \cite{veretennikov1981strong,zvonkin1974transformation} put in light the remarkable fact according to which the well-posedness of the ODE can be obtained under very week conditions on $b$ by adding a random force to the system, which then becomes the following SDE :
\begin{equation}
\label{eq:Intro_SDE}
X_t = x_0 + \int_0^t b(s,X_s) ds + \sigma W_t, \quad t \in [0,T], \quad x_0 \in \real^d,
\end{equation}
with $\sigma>0$ and $W$ a Brownian motion on $\real^d$ (we use the notation $X$ to stress than the solution is not deterministic anymore). This phenomenon is usually referred to \textit{regularization by noise effect} or \textit{stochastic regularization}. To be more precise, pathwise uniqueness can be obtained for Equation (\ref{eq:Intro_SDE}) for any vector field $b$ satisfying weak regularity conditions : a boundedness assumption (\cite{veretennikov1981strong}) or a Ladyzhenskaya-Prodi-Serrin (LPS) type condition (see \cite{krylov2005strong}) $b \in L^q([0,T];L^p(\real^d))$ :
\begin{equation}
\label{eq:LPS}
\frac{d}{p} + \frac{2}{q} < 1, \quad p,q \geq 2.
\end{equation}
In addition, this result can be captured and quantified by the so-called \textit{It\^o-Tanaka trick} or \textit{Zvonkin's tranform} (\cite{flandoli2010well}) which reads as follows : 
\begin{equation}\label{eq:ItoTanakaSemiMat1}
\int_0^T b(t,X_t+x) dt = -F(0,X_0+x) - \int_0^T \nabla F(t,X_t+x) \cdot dW_t,
\end{equation}
and which relates the process $X$ to the solution $F:[0,T]\times \real^d \to \real^d$ of the parabolic system of PDEs
\begin{equation}
\label{eq:intro_PDE}
\left\{\begin{array}{ll}
\partial_t F(t,x) +  \mathcal{L}^X F(t,x) = b(t,x), \quad \forall (t,x) \in [0,T)\times \real^d,
\\ F(T,x)=0,\quad \forall x\in\real^d,
\end{array}\right.
\end{equation} 
with $\mathcal{L}_s^X \Phi(x):= b(s,x) \cdot \nabla \Phi(x) + \frac12 \Delta \Phi(x)$. Indeed, one can prove (see for a precise statement \cite{flandoli2010well,krylov2005strong}) that the solution $F$ to the PDE admits two weak derivatives in space and one in the time variable which entails that for any positive time $t$, the mapping 
\begin{equation*}
x\mapsto \int_0^t b(s,X_s+ x) ds
\end{equation*}
is more regular than the field $b$ itself (recall Relation (\ref{eq:ItoTanakaSemiMat1})).   
%
\\\\ 
\noindent
Note that investigating such regularization effect for ODEs finds interest in fluid mechanics equations which take the form of (non-linear) transport PDEs (we refer to \cite{Beck_Flandoli_Gubinelli_Maurelli} for a survey on that account). For that purpose, the LPS condition (\ref{eq:LPS}) provides a natural framework in which fits this paper. However determining if the counterpart of the previous paradigm for ODEs transfers to non-linear transport PDEs is valid or not is mainly an open question. Although, most references in the literature, where regularization effects for SDEs are obtained, are based on the It\^o-Tanaka trick it does not constitute the only technic for that regard (see for instance  \cite{Beck_Flandoli_Gubinelli_Maurelli,catellier2016averaging}).\\\\
\noindent
In this paper we investigate a general framework in which the It\^o-Tanaka trick is valid. Indeed, at this stage, one can point out at least two limitations to Relation (\ref{eq:ItoTanakaSemiMat1}). First, the strong link to the PDE (\ref{eq:intro_PDE}) seems to be bound to the semimartingale realm (where one relates an SDE as a probabilistic counterpart of a parabolic PDE using the It\^o formula). Another limitation is to investigate if Relation (\ref{eq:ItoTanakaSemiMat1}) can be extended to random fields $b$. Note that this step seems somehow mandatory to study the (possible) regularization phenomenon for a class of fluid mechanics equations which takes the form of non-linear transport PDEs (we refer to the comment \cite[page 6]{flandoli2010well} on that question). For instance, counter-examples can be derived in the case where $b$ is random as this extra randomness can cancel the effect of the noise $W$. As an example, consider $b:[0,T] \times \real^d \to \real^d$ a non-smooth deterministic field, and $\tilde b:\Omega\times [0,T] \times \real^d \to \real^d$ defined as : $\tilde b(\omega,t,x):=b(t,x-\sigma W_t(\omega))$, then it is clear that SDE 
$$ dX_t = \tilde b(t,X_t) dt +\sigma dW_t, $$
is equivalent to the deterministic ODE (by setting $x_t:=X_t-\sigma W_t$) : 
$$ dx_t = b(t,x_t) dt.$$
This example enlights the fact that somehow the randomness and space variables $(\omega,x)$ have to be decoupled for a relation of the form (\ref{eq:ItoTanakaSemiMat1}) to be in force. In \cite{Duboscq_Reveillac}, the authors have extended the It\^o-Tanaka trick to that framework, for which the improvement of regularity is obtained if the field $b$ is Malliavin differentiable. In particular, this extra randomness is harmless for the regularity in the space variable for $b$ if $(\omega,x)$ are "decoupled".\\\\    
\noindent
In this paper we revisit the It\^o-Tanaka trick for random fields $b$ and a non-semimartingale driving noise. More specifically, we bound ourselves to the case of a fractional Brownian motion (fBm) noise which allows one to compare our results with for instance the work \cite{catellier2016averaging} in which pathwise uniqueness is proved for SDEs of the form (with $W$ replaced by a fBm) but without using the It\^o-Tanaka trick. Our approach is based on the use of Malliavin calculus arguments allowing one to escape the semimartingale context and to consider random fields $b$. To illustrate our key argument, we provide informal computations in the following particular example : $d=1$, $b:\real \to \real$ (so $b$ is deterministic and does not depend on the time variable). We stress that our main result is valid in any finite dimension and for a time-dependent vector field $b$, which is random (more precisely adapted according to assumptions presented in Section \ref{section:main}). Consider once again the solution $X$ to the SDE (\ref{eq:Intro_SDE}), and let $(P^X_t)_{t\geq0}$ the transition operator associated to it. For any fixed time $t>0$, assuming that the random variable $A_t = b(t,X_t + x)$ is square integrable, one can apply the Clark-Ocone formula (which will be recalled below as Relation (\ref{eq:CO})) to get 
\begin{equation*}
A_t = \mathbb{E}\left[ A_t \middle| \mathcal{F}_0\right] + \int_0^t \mathbb{E}\left[ D_s A_t \middle| \mathcal{F}_s\right] dW_s,
\end{equation*}
where $D$ denotes the Malliavin derivative (which will also be recalled in the next section). Hence, very formally, integrating with respect to $t$, we obtain : 
\begin{align*}
\int_0^T b(t,X_t+x) dt &= \int_0^T P^X_t b(t,X_0 + x) dt + \int_0^T \int_0^t D_s P^X_{t-s} b(t,X_s+x)dW_s dt
\\ &= \int_0^T P^X_t b(t,X_0 + x) dt + \int_0^T \int_0^t \frac{\partial}{\partial x} P^X_{t-s} b(t,X_s+x)dW_s dt
\\ &= \int_0^T P^X_t b(t,X_0 + x) dt + \int_0^T \int_s^T \frac{\partial}{\partial x} P^X_{t-s} b(t,X_s+x) dt dW_s,
\end{align*}
where we have used stochastic Fubini's theorem. This relation exactly  matches with the It\^o-Tanaka trick (\ref{eq:ItoTanakaSemiMat1}) as the mild solution $F$ to the PDE (\ref{eq:intro_PDE}) writes down as : 
\begin{equation}
\label{eq:introDuhamel}
F(s,x) =-\int_s^T P_{t-s} ^X b(t,x) dt, \quad t\in [0,T], \quad x \in \real^d.
\end{equation}

From these simple and very formal computations, one can make several remarks. First, the regularization effect is contained in the form of the solution to the PDE (using the semigroup associated to $X$). Then, this approach seems restricted to the deterministic case, as a measurability issue would prevent one to define the stochastic It\^o integral $\int_0^T \int_s^T \frac{\partial}{\partial x} P^X_{t-s} b(t,X_s+x) dt dW_s,$ even in the case of an adapted random field $b$. This problem has been solved in \cite{Duboscq_Reveillac} where the PDE has to be replaced by a Backward Stochastic PDE whose solution is explicitly given as the predictable projection of the solution to the PDE (\ref{eq:intro_PDE}). However, BSPDEs can only be solved and studied in a semimartingale context. The main idea of this paper is to use the classical representation of a fBm as the It\^o integral of a well-chosen kernel against a standard Brownian motion, and to the apply (several times) the Clark-Ocone formula to a functional of the form (\ref{eq:introDuhamel}). This functional will not be a solution to a PDE (or a BSPDE) which fits with the well-known result according to which the fBm cannot be related to a Markov semi-group, but it somehow plays this role. The several use of the Clark-Ocone formula allows us to precisely take into account the randomness coming from the field $b$ and from the noise. Hence we obtain a generalization of the It\^o-Tanaka trick as Theorem \ref{thm:main}. We apply this result to recover the well-posedness of the fractional SDE associated to $b$ in Theorem \ref{thm:CauchySDE}. \\\\
\noindent Finally, we would like to make a comment on the reference \cite{catellier2016averaging} where the authors prove the well-posedness of the fractional SDE. The proof relies on two ingredients: the study of the Fourier transform of the occupation measure related to $W$ (to be more specific, on the $(\rho,\gamma)$-irregular property of $W$) and the reformulation of the SDE as a Young-type ODE where the time-integral of the drift is reinterpreted as a Young integral. The $(\rho,\gamma)$-irregular property of $W$ provides the regularization effects of $W$ and the authors do not rely on the It\^o-Tanaka trick but on a kind of discrete martingale decomposition and a Hoeffding lemma. We remark that this martingale decomposition possesses some similarities with the Clark-Ocone formula. In Section \ref{section:SDE}, we follow the same reformulation (and the argument to construct the Young integral) to prove the existence and uniqueness of a fractional SDE but we do not prove exactly the $(\rho,\gamma)$-irregular property since we rely on more straightforward strategy in Sobolev spaces (at the cost of an embedding to recover estimates in H\"older spaces).
\\\\
\noindent
We proceed as follows. In the next section we present the main notations. The main result (Theorems \ref{thm:main}) is presented in Section 3. The application to uniqueness of fractional SDEs (with additive noise) with adapted coefficients is presented in Section \ref{section:SDE}. The proof of Theorem \ref{thm:main} is postponed to Section \ref{section:proofmain}.  

\section{Notations and preliminaries}

\subsection{General notations}

Throughout this paper $T$ denotes a positive real number, $\lambda$ stands for the Lebesgue measure and $\mathcal B(E)$ denotes the Borelian $\sigma$-field of a given measurable pace $E$. We set also $\mathbb{N}^*$ the set of integers $n$ with $n\geq 1$.\\\\
For any $x$ in $\real^d$, we denote by $x_k$ the $k$-th coordinate of $x$ that is $x=(x_1,\ldots,x_d)$.\\\\
\noindent
For any $r,\ell \in \mathbb{N}^*$, we denote by $\mathcal{C}^{r}(\real^\ell)$ the set of $r$-times continuously differentiable (real-valued) mappings defined on $\real^d$. We also let $\mathcal{C}_c^{\infty}(\real^\ell)$ the set of infinitely differentiable mappings with compact support.\\\\ 
\noindent
Let $\varphi:\real^d \to \real$ belongs to $\mathcal{C}^{k}(\real^d)$, $\ell\in \mathbb{N}^*$, $n_1,\ldots,n_\ell, p_1,\ldots,p_\ell$ in $\mathbb{N}$ with $\sum_{i=1}^\ell n_i^{k_i} = n$, we denote by $\frac{\partial^k \varphi}{\prod_{i=1}^\ell \partial{{x_i}^{k_i}}}$ the partial derivative of $\varphi$ with respect to the variables $x_i$ with order $k_i$. $\nabla \varphi$ will refer to the gradient of $\varphi$. 
Finally for any $x$ and $h$ in $\real^d$, we write $\nabla^k \varphi(x) \cdot h^k$ the action of the $k$-order differentiable of $\varphi$ (noted $\nabla^k \varphi(x)$) on $h^k:=(h,\ldots,h)$. Finally, we denote by $\Delta$ the Laplacian operator.\\\\
\noindent
For $p,m\in\mathbb{R}$, we set 
\begin{equation*}
W^{m,p}(\mathbb{R}^d) = \left\{ \varphi\in L^p(\mathbb{R}^d);  \mathcal{F}^{-1}\left(([1+|\xi|^2]^{m/2}\hat{\varphi}\right)\in L^p(\mathbb{R}^d)) \right\},
\end{equation*}
the usual Sobolev spaces equipped with its natural norm
\begin{equation*}
\|\varphi\|_{W^{m,p}(\mathbb{R}^d)}:= \left\|\mathcal{F}^{-1}\left(([1+|\xi|^2]^{m/2}\hat{\varphi}\right)\right\|_{L^p(\mathbb{R}^d))},
\end{equation*}
where $\hat{\varphi}(\xi) = \mathcal{F}(\varphi)(\xi)$ and $\mathcal{F}$ (resp. $\mathcal{F}^{-1}$) denotes the Fourier transform (resp. the inverse Fourier transform).\\

\noindent We also make use of the following notation : let $(\mathcal{E},\mathcal{B},\mu)$ be a mesured space and $(G,\|\cdot\|_{G})$ be a Banach space, and $r\geq 0$. We denote by $L^r(\mathcal{E};G)$ the space of measurable mappings $\varphi : \mathcal{E} \to G$ with  
$$ \|\varphi\|_{L^r(\mathcal{E};G)}^{r}:=\int_{\mathcal{E}} \|\varphi(y)\|_{G}^r \; \mu(dy) <+\infty. $$
Depending on the context, the definition of the integral will be made precise.

\subsection{The fractional Brownian motion}

Let $(\Omega,\mathcal{F},\P)$ be a probability space, $d\in \mathbb{N}$ ($d\geq 1$) and $B:=(B_1(s),\ldots,B_d(s))_{s\in (-\infty,T]}$ a standard $\real^d$-valued two-sided Brownian motion (with independent components). We set $\left(\mathcal F_t\right)_{t\in (-\infty,T]}$ the natural (completed and right-continuous) filtration of $B$. We assume for simplicity that $\mathcal{F} = \sigma\left(B(s), \; s\in (-\infty,T]\right)$.\\\\
\noindent  
More generally, for any $\real^d$-valued stochastic process $(X(t))_{t\in (-\infty,T]}$ we will denote by $X^j$ the $j$th component of $X$.\\\\
\noindent
The main object of our analysis will be $d$-dimensional fractional Brownian motion 
$$W^H:=(W_1^{H}(s),\ldots,W_d^{H}(s))_{s\in [0,T]},$$ defined as 
$$ W_j^{H}(s) = \int_{-\infty}^s \left((s-u)_{+}^{H-1/2} - (-u)_+^{H-1/2} \right)dB_{j}(u), \quad s\in [0,T], \quad j\in \{1,\cdots,d\}$$
where $H$ is a given parameter in $(0,1)\setminus \left\{\frac 12\right\}$. A crucial decomposition is on analysis relies on the following split of the fBm $W^H$ as follows : 
\begin{align}
\label{eq:decompositionfBm}
W_j^{H}(s) &= \int_{-\infty}^s \left[(s-u)_{+}^{H-1/2} - (-u)_+^{H-1/2} \right]dB_{j}(u) \nonumber
\\&= \int_{t}^s (s-u)^{H-1/2}dB_{j}(u) + \int_{-\infty}^t \left[(s-u)_{+}^{H-1/2} - (-u)_+^{H-1/2} \right]dB_{j}(u) \nonumber
\\&=: W_j^{1,H}(t,s) + W_j^{2,H}(t,s),
\end{align}
Note that for a given $(s,t)$ with $t<s$, the random variable $W_j^{1,H}(t,s)$ is independent of $\mathcal{F}_t$ whereas the process $W^{2,H}:=(W^{2,H}(t,s))_{t\in[0,s]}$ is $\left(\mathcal{F}_t\right)_{t\in[0,s]}$-adapted. It is worth noting that this decomposition is somehow natural in the context of stochastic regularisation and was already used in \cite{catellier2015rough} as only the component $W^{1,H}$ contributes to the regularising effect we will describe in the next sections. \\\\  
\noindent
We now turn to the notion of (smooth) adapted random field. 

\begin{defi}[(smooth) adapted random field]
\label{definiton:randomfield}
\textit{}\\
\begin{itemize}
\item[(i)] A random field is a $\mathcal F_T\otimes\mathcal{B}(\real^d)$-measurable mapping $\varphi:\Omega\times\real^d \to \real$.
\item[(ii)] An adapted random field is a $\mathcal F_T\otimes \mathcal{B}([0,T])\otimes\mathcal{B}(\real^d)$-measurable mapping $\varphi:\Omega\times[0,T]\times\real^d \to \real$ such that for any $x$ in $\real^d$, $\varphi(\cdot,x)$ is $\left(\mathcal{F}_t\right)_{t\in[0,T]}$-adapted.
\item[(iii)] A smooth adapted random field is an adapted random field $\varphi$ such that $x\mapsto \varphi(\omega,t)$ is infinitely continuously differentiable with bounded derivatives of any order for $\lambda\otimes\P$-a.e. $(\omega,t)$ in $[0,T]\times \Omega$.
\end{itemize}
\end{defi}

We denote by $P:=(P_t)_{t\in[0,T]}$ the Heat semigroup. For simplicity,  we will use throughout this paper, the following notation for the conditional expectation.  

\begin{notations}
\label{notations:main1}
For $t$ in $[0,T]$, we set $\E_t[\cdot] := \E[\cdot \vert \mathcal F_t]$.\\
\end{notations}

\subsection{Malliavin-Sobolev spaces}

In this section, we introduce the main notations about the Malliavin calculus for random fields.

\begin{defi}
\label{definition:cylind}
\begin{itemize}
\item[(i)] Consider $\mathcal{S}_{r.v.}$ be the set of cylindrical random variables, that is the set of random fields $F:\Omega \times \real^d \to \real$ such that there exist : 
$$ n \in \mathbb{N}^*, \; 0\leq \gamma_1<\gamma_2<\cdots<\gamma_n\leq T, \quad \varphi:\real^n\times\real^d \to \real \in \mathcal{C}_c^\infty(\real^{n+d})$$
such that 
\begin{equation}
\label{eq:cylind}
F(\omega,x)=\varphi(B(\gamma_1),\cdots,B(\gamma_n),x), \quad \omega \in \Omega, \; x\in \real^d.
\end{equation}
\item[(ii)] The set of cylindrical random fields denoted by $\mathcal S$, consists of random fields $F:[0,T] \times \Omega \times \real^d \to \real$ such that there exist : 
$$ n \in \mathbb{N}^*, \; 0\leq \gamma_1<\gamma_2<\cdots<\gamma_n\leq T, \quad \varphi:[0,T]\times (\real^d)^n\times\real^d \to \real $$
such that 
\begin{equation}
\label{eq:cylindP}
F(\omega,t,x)=\varphi(t,B(\gamma_1),\cdots,B(\gamma_n),x), \quad \omega \in \Omega, \; x\in \real^d, \quad t\in [0,T],
\end{equation}
where $\varphi(t,\cdot) \in \mathcal{C}_c^\infty\left((\real^d)^n\times\real^d\right), \; \forall t \in [0,T], $
and 
$$ \sup_{t\in[0,T]} \left(\|\varphi(t,\cdot)\|_\infty+\|\mathcal{L} \varphi(t,\cdot)\|_\infty\right) <+\infty $$
with $\mathcal{L}$ any partial derivative of any order.
\item[(iii)] The set of adapted cylindrical random fields denoted by $\mathcal{S}_{ad}$, consists of adapted random field is a random field $F:[0,T] \times \Omega \times \real^d \to \real$ such that there exist : 
$$ n \in \mathbb{N}^*, \; 0\leq \gamma_1<\gamma_2<\cdots<\gamma_n\leq T, \quad \varphi:[0,T]\times (\real^d)^n\times\real^d \to \real $$
such that 
\begin{equation}
\label{eq:cylindPad}
F(\omega,t,x)=\varphi(t,B(\gamma_1\wedge t),\cdots,B(\gamma_n\wedge t),x), \quad \omega \in \Omega, \; x\in \real^d, \quad t\in [0,T],
\end{equation}
where $\varphi(t,\cdot) \in \mathcal{C}_c^\infty\left((\real^d)^n\times\real^d\right), \; \forall t \in [0,T], $
and 
$$ \sup_{t\in[0,T]} \left(\|\varphi(t,\cdot)\|_\infty+\|\mathcal{L} \varphi(t,\cdot)\|_\infty\right) <+\infty $$
with $\mathcal{L}$ any partial derivative of any order.
\end{itemize}
\vspace{1em}
Obviously, $\mathcal{S}_{r.v.} \subset \mathcal{S}$ and $\mathcal{S}_{ad} \subset \mathcal{S}$.
\end{defi}

We now define the Malliavin derivative of any adapted random field $F$ in $\mathcal S$.

\begin{defi}
\label{definition:Malliavin}
Let $F$ in $\mathcal S$ with representation (\ref{eq:cylindP}). Then, we define the Malliavin gradient $DF$ of $F$ as follows : 
$$ DF : [0,T]\times\Omega\times \real^d \to L^p([0,T];\real^d) $$
with for any $j$ in $\{1,\cdots,d\}$,
$$ \left(D_jF(t,\omega,x)\right)(u):=\sum_{i=1}^n \frac{\partial \varphi}{\partial y_i}(B(\gamma_1)(\omega),\cdots,B(\gamma_n)(\omega),x) \textbf{1}_{[0,\gamma_i]}(u), \quad u\in [0,T], \; \omega\in \Omega, x\in \real^d. $$  
\end{defi}

We can now define Malliavin-Sobolev spaces associated to the Malliavin and the spatial derivatives for random fields. 

\begin{defi}
\label{defi:Malliavin-Sobolevspaces}
Set $m\in\mathbb{R}$.
\begin{itemize}
\item[(i)] We set $\D^{1,m,p}$ the closure of $\mathcal{S}_{r.v.}$ with respect to the seminorm $\|\cdot\|_{\D^{1,m,p}}$ with
\begin{equation}\label{eq:NormD}
 \|F\|_{\D^{1,m,p}}^p:= \E[\|F\|_{W^{m,p}}^p] + \int_0^T \E\left[\|D_\theta F\|_{W^{m,p}(\mathbb{R}^d)}^p\right] d\theta.
\end{equation}
\item[(ii)] We set $\D_p^{1,m,p}$ the closure of $\mathcal{S}$ with respect to the seminorm $\|\cdot\|_{\D_p^{1,m,p}}$ with
\begin{equation}\label{eq:NormDq}
 \|F\|_{\D_p^{1,m,p}}^p:= \int_0^T \|F(t,\cdot)\|_{{\D^{1,m,p}}}^p dt.
\end{equation}
\end{itemize}
\end{defi} 

This definition, requires some justifications. Indeed, note that $D \nabla^k F = \nabla^k D F $ for $F$ in $\mathcal{S}_{r.v.}$. In addition, as proved in \cite[Lemma Appendix A.1 and Lemma Appendix A.2]{Duboscq_Reveillac}, the operators $D \nabla^k $ (and so $\nabla^k D$) are closable from $\mathcal{S}$ to $L^p([0,T]\times \Omega \times \real^d; \real^d)$.

\begin{remark}
By definition $$\mathcal{S}_{ad} \subset {\D_p^{1,m,p}}, \quad \forall m\geq 0, \; p\geq 2 $$
\end{remark}

\noindent
We conclude this section with two properties of the Malliavin derivative. 

\begin{lemma}
\label{lemma:propMalliavin}
\begin{itemize}
\item[(i)] (Chain rule). Let $F$ be in $\mathcal{S}$ and $G$ be in $\mathcal{S}_{r.v.}$. Then, for any $t$ in $[0,T]$, $F(t,G)$ belongs to $\mathcal{S}_{r.v.}$ and :
$$ (D_j F(t,G))(u) = (D_j F(t,x))(u))_{x=G} + \frac{\partial F}{\partial x_j}(t,G) \times (D_j G)(u), \quad j \in \{1,\cdots,d\}, \; u\in [0,T].$$
\item[(ii)] Let $m$ a real number, $p\geq 2$, $t$ in $[0,T]$ and $G$ be in $\D^{1,m,p}$. If $G$ is $\mathcal{F}_t$-measurable, then for any $j \in \{1,\cdots,d\}$ :
$$ ((D_j G)(s))_{s>t} = 0, \quad \textrm{where the equality is understood in } L^2(\Omega\times(t,T])). $$ 
\end{itemize}
\end{lemma}

\subsection{Clark-Ocone formula}

Let $\mathcal{S}_{r.v.}$ be the set of random variables of the form $F=\varphi(B(t_1),\cdots,B(t_n))$ in $\mathcal{S}$ (that is that do not depend on the $x$-variable). We start with the following lemma whose proof can be found for instance in \cite{Nualartbook,Privault_LectureNotes}.

\begin{lemma}
\label{lemmaExten}
The operator 
$$
\left\lbrace
\begin{array}{lll}
\Dp : &\mathcal{S}_{r.v.} &\to L^2([0,T]\times \Omega; \real^d)\\
&F &\mapsto (\E_s[(DF)(s)])_{s\in[0,T]}, 
\end{array}
\right.
$$
is continuous with respect to the $L^2(\Omega)$-norm. In particular in extends to $L^2(\Omega)$.
\end{lemma}

Consider $F:\Omega \to \real$ a random variable with $\E[|F|^2]<+\infty$. Then for any $t$ in $[0,T]$,
\begin{equation}
\label{eq:CO}
F =\E_t[F] + \sum_{j=1}^d \int_t^T \E_s[(D_j F)(s)] dB_j(s)
\end{equation}

$$ \left(\Dp F\right)(u) := \E_u[(DF)(u)] $$ 

Note that by Lemma \ref{lemmaExten}, the operator $(\E_s[D_s F])_s$ is well-defined even though $F$ is not Malliavin differentiable. 


\section{Main result}
\label{section:main}

\begin{assumption}
\label{assumption:main}
Let $m\in \mathbb{R}$, $p\geq 2$ and $\alpha$ in $\real$. An adapted random field $f:[0,T]\times\Omega\times\real^d \to \real$ is said to enjoy Assumption \ref{assumption:main} if : 
$$
1/2-H\alpha-1/p>0 \quad \textrm{ and } \quad f \in \mathbb{D}^{1,m-\alpha,p}_p.
$$ 
\end{assumption}

\noindent
We set : 
\begin{notations}
\label{notations:main}
\begin{itemize}
\item[(i)] Given an adapted smooth random field $f$, we set for $0\leq t\leq u \leq s\leq T,\; x\in \real^d, \; j\in \{1,\cdots,d\}$ :
\begin{equation}
\label{eq:notationproof}
f^a(s,t,x) = \E_{t}\left[f(s,x)\right] \quad \textrm{ and } \quad g_j(s,u,x) =(\Dp f(s,x))_j (u) =  \E_u\left[(D_j f(s,x))(u)\right]
\end{equation}
\item[(ii)] For fixed $x$ in $\real^d$, let
\begin{equation}
\label{eq:definitionF}
F(t) := \int_t^T P_{\frac{1}{2H}(s-t)^{2H}}f(s,W^{2,H}(t,s)+x)ds, \quad t\in [0,T],
\end{equation}
\begin{equation}
\label{eq:definitionFa}
F^a(t) := \E_t\left[\int_t^T P_{\frac{1}{2H}(s-t)^{2H}}f(s,W^{2,H}(t,s)+x)ds\right], \quad t\in [0,T]
\end{equation}
\end{itemize}
\end{notations}

With these notations at hand we can state a non-semimartingale counterpart of the It\^o-Tanaka-Wentzell trick for as:
\begin{theorem}\label{thm:main}
Let $f:[0,T]\times\Omega\times\real^d \to \real$ be an adapted random field and $(p,m,\alpha)$ such that Assumption \ref{assumption:main} is in force.Then, $\forall t\in[0,T]$, we have
\begin{align}
\label{eq:mainformula}
\int_0^t f(r,W^H_r+x)dr =&  \int_0^t P_{\frac{1}{2H}r^{2H}} f(r,W^{H}(r)+x)dr\nonumber
\\ & +\sum_{j=1}^d \int_0^t \int_u^t P_{\frac{1}{2H}(r-u)^{2H}} \frac{\partial}{\partial x_j} f^a(r,u,W^{2,H}(u,r)+x)(r-u)^{H-1/2}dr dB_j(u)\nonumber
\\ & +\sum_{j=1}^d \int_0^t  \int_{u}^t P_{\frac{1}{2H}(r-u)^{2H}} \frac{\partial}{\partial x_j} g_j(r,u,W^{2,H}(u,r)+x)(r-u)^{H-1/2}drdu\nonumber
\\& - \sum_{j=1}^d \int_0^t  \int_{u}^tP_{\frac{1}{2H}(r-u)^{2H}} g_j(r,u,W^{2,H}(u,r)+x)(r-u)^{H-1/2}dr dB_j(u),
\end{align}
where the equality holds in $L^\infty([0,T];L^p(\Omega; W^{m,p}(\real^d)))$.
\end{theorem}

\begin{remark}\label{rem:MallDiv}
Note that the second term in the right-hand side of Formula (\ref{eq:mainformula}) rewrites as :
$$ \int_s^t \int_u^t P_{\frac{1}{2H}(r-u)^{2H}} \nabla f^a(r,u,W^{2,H}(u,r)+x)(r-u)^{H-1/2}dr \cdot dB(u),$$
whereas the third term is some sort of divergence term with respect to both the Malliavin derivative and the usual spatial derivative. More precisely, if we define $\textrm{div}^{(\omega,x)}$ this joint divergence operator (applied to a random field $F:\Omega\times \real^d$) as : 
$$ \left(\textrm{div}^{(\omega,x)} F\right)(u):= \sum_{j=1}^d \frac{\partial}{\partial x_j} \E_{u}\left[D_j(F(\cdot,x))(u)\right],$$
then the third term rewrites as 
$$ \int_s^t  \int_{u}^t (r-u)^{H-1/2} P_{\frac{1}{2H}(r-u)^{2H}} \left(\textrm{div}^{(\omega,x)} f(r,y)\right)(u)_{|y=W^{2,H}(u,r)+x}drdu $$
\end{remark}

We postpone the proof of this result to Section \ref{section:proofmain}.

\section{Application to fractional SDEs}
\label{section:SDE}

In this section, we use Theorem \ref{thm:main} to obtain new results concerning the existence and uniqueness of SDEs with singular drifts and additive fractional Brownian motions. Our result applies in fact to a reformulation of such SDEs as Young ODEs and we state some key results around these equations.

\subsection{Main result}

We consider the following SDE
\begin{equation}\label{eq:EDS}
X_t = X_0 + \int_0^t b(s,X_s) ds + W^H_s,
\end{equation}
where $b\,:\; \Omega\times\mathbb{R}^+\times \mathbb{R}^d\to \mathbb{R}^d$ is an adapted (generalized) function and $(W^H_t)_{t\geq 0}$ a fractional Brownian motion of Hurst index $H\in (0,1)$. By making the following change of variable
\begin{equation*}
Y_t : = X_t - W^H_t,
\end{equation*}
and setting, $\forall (u,x)\in\mathbb{R}^+\times\mathbb{R}^d$,
\begin{equation}\label{eq:defA}
A_{u}(x) = \int_0^u b(s,x + W^H_s) ds,
\end{equation}
we can relate \eqref{eq:EDS} to the following Young type ODE
\begin{equation}\label{eq:EDSbis}
Y_t = Y_0 + \int_0^t A_{ds}(Y_s),
\end{equation}
where the integral is understood as a nonlinear generalization of the Young integral, $\forall Z\in\mathcal{C}^{\gamma}([0,T];\mathbb{R}^d)$,
\begin{equation*}
\int_0^t A_{ds}(Z_s) = \lim_{|\Pi_{[0,t]}|\to 0} \sum_{[u,v]\in \Pi_{[0,t]}} \delta A_{u,v}(Z_u),
\end{equation*}
with
\begin{equation*}
\delta A_{u,v}(x) := A_{v}(x) - A_{u}(x),
\end{equation*}
and $\Pi_{[0,t]}$ denoting a discretization of $[0,t]$. Before stating our result, we need the following "chain rule" assumption on the Malliavin derivative of $b$.

\begin{assumption}
\label{asm:chainrule}
Let $\ell \geq 2$, $q,p\in[1,+\infty]$, $k\in\mathbb{R}$, $\iota\in [0,1]$ and $\sigma,\bar{\sigma}\in[\ell,\infty]$ such that 
\begin{equation}
\frac1{H}\left(\frac12 - \frac2{\ell}-\frac1{q}\right)+ k-1-\frac dp >0 \quad\textrm{and}\quad \frac1{\sigma} + \frac1{\bar{\sigma}} =\frac1{\ell}.
\end{equation}
We assume that $b$ is an adapted function which belongs to $L^\ell(\Omega; L^q([0,T];W^{k,p}(\mathbb{R}^d)))$ and that:
\begin{enumerate}
\item[i)] there exist a function a function $b'\in L^{\sigma}(\Omega ; L^{q}([0,T]; W^{k-\iota,p}(\mathbb{R}^{d\times d})))$ and a mapping  $\upsilon \in L^{\bar{\sigma}}(\Omega; L^\infty([0,T]\times\mathbb{R}^{d\times d}))$ such that
\vspace{-0.5em}
\begin{align*}
D_{\theta} b(t,x) &= b'(t,x) \upsilon(\theta,t), \quad \forall\theta\leq t\leq T,\, \P-a.s.,
\end{align*}
where, $\forall t\in[0,T]$, $b'(t,x)$ is $\mathcal{F}_t$-adapted for any $x\in\mathbb{R}^d$ and $\upsilon(\theta,t)$ is a $\mathcal{F}_t$ adapted function for any $0\leq\theta\leq t$,
\item[ii)]  there exists $C_{1}\in L^{\bar{\sigma}}(\Omega;\mathbb{R}^{+*})$ such that one of the following statement is in force
\begin{itemize}
\item for any $0\leq \theta\leq s\leq t\leq T$,
\begin{equation}\label{eq:asmHolderb}\begin{array}{cc}
|\upsilon(\theta,t)| \leq C_{1}|\theta - t|^{H\iota},
\end{array}\end{equation}
\item $\iota = 0$ and $\upsilon(\theta,t) = C_1 \bold{1}_{\{\theta\leq \tau_b\}}$ where $\tau_b$ is a random variable with values in $[0,t]$,
\end{itemize}
\end{enumerate}
\end{assumption}

We can now give our result.
\begin{theorem}\label{thm:CauchySDE}
Let $T>0$. Under Assumption \ref{asm:chainrule} (see below), there exists $\beta > 1/2$ such that Equation (\ref{eq:EDSbis}) admits a unique solution $Y\in\mathcal{C}^{\beta}([0,T];\mathbb{R}^d)$.
\end{theorem}

\begin{remark}\label{rmk:bndAssm}
The equality obtained in Theorem \ref{thm:main} holds in $L^{\infty}([0,T];L^p(\Omega;W^{m,p}(\mathbb{R}^d)))$. However, in the proof of Theorem \ref{thm:CauchySDE}, we  bound an increment of each term in $L^{\ell}(\Omega;W^{m,p}(\mathbb{R}^d))$. That is why we need the stronger Assumption \ref{asm:chainrule}.
\end{remark}

\begin{remark}
Even though $b$ might be defined in the sense of generalized functions (or Schwarz distribution), the Young integral \eqref{eq:EDSbis} can still be well-defined due to regularization effect of $(W^H_t)_{t\geq0}$ whereas the integral of the drift in \eqref{eq:EDS} does not make sense. Nevertheless, it is possible to define a notion of "controlled solution" for \eqref{eq:EDS} (see  \cite{catellier2016averaging}).
\end{remark}

\subsection{The Cauchy problem for Young ODEs}

We recall here some results on the nonlinear Young integration procedure and the Cauchy problem related to the Young ODE. Here, we simply give the results from \cite{catellier2016averaging} but the reader might also be interested in \cite{gubinelli2004controlling,friz2010multidimensional,lyons2007differential}. 

\begin{defi}
Let $T>0$, $\beta,\gamma\in (0,1]$, $I = [0,T]$ and $V,W$ to Banach spaces. For all $n\in\mathbb{N}$, and any mapping $A: I\times V\to W$, we define the norm
\begin{equation*}
\|A\|_{\beta,\gamma} = \sup_{\substack{s,t\in[0,T]\\s\neq t}} \sup_{\substack{x,y\in V\\x\neq y}} \frac{|\delta A_{s,t}(x) - \delta A_{s,t}(y)|_W}{|t-s|^{\beta}|x-y|^{\gamma}_V},
\end{equation*}
and
\begin{equation*}
\|A\|_{\beta,n+\gamma} = \|\mathfrak{D}^n A\|_{\beta,\gamma}+\sum_{k = 0}^n  \sup_{\substack{s,t\in[0,T]\\s\neq t}} \sup_{x\in\mathbb{R}^d}\frac{|\mathfrak{D}^n \delta A_{s,t}(x)|_{\mathcal{L}^k(V;W)}}{|t-s|^{\beta}},
\end{equation*}
where $\mathfrak{D}$ denotes the Fr\'echet derivative from $V$ to $W$.
\end{defi}

We can now proceed to state the results from \cite{catellier2016averaging}. The first result concerns the existence of the nonlinear Young integral. 
\begin{theorem}
Let $\beta,\gamma,\rho>0$ with $\beta + \gamma\rho>1$, $V,W$ two Banach spaces and $I$ a finite interval of $\mathbb{R}$. We consider $A\in \mathcal{C}^{\beta,\gamma}(I,V;W)$ and $Y\in\mathcal{C}^{\rho}(I;V)$. For any $s,t\in I$ such that $s\leq t$, the following nonlinear Young integral exists and is independent of the partition
\begin{equation*}
\int_s^t A_{dr}(Y_r) := \lim_{\substack{\Pi\;\textrm{partition of}\;[s,t]\\ |\Pi|\to 0}} \sum_{[u,v]\in \Pi}\delta A_{u,v}(Y_u).
\end{equation*}
Furthermore, we have
\begin{enumerate}
\item for all $u\in [s,t]$, the equality
\begin{equation*}
\int_s^t A_{dr}(Y_r) = \int_s^u A_{dr}(Y_r) +  \int_u^t A_{dr}(Y_r),
\end{equation*}
\item the following bound
\begin{equation*}
\left| \int_s^t A_{dr}(Y_r) - \delta A_{s,t}(Y_s) \right|_W \lesssim_{\beta,\gamma,\rho} \|A\|_{\beta,\gamma}  \|Y\|_{\mathcal{C}^{\rho}(I;V)}^{\gamma}(t-s)^{\beta+\gamma\rho}
\end{equation*}
\item for all $s,t\in I$ such that $s\leq t$ and $R>0$, the map
\begin{equation*}
(Y,A) \to \int_s^t A_{dr}(Y_r)
\end{equation*}
is a continuous function from $(\{Y\in\mathcal{C}^{\rho}(I;V);\; \|Y\|_{\mathcal{C}^{\rho}(I;V)}\leq R\},\|\cdot\|_{L^{\infty}([s,t];V)})\times (\mathcal{C}^{\beta,\gamma}(I,V;W),\|\cdot\|_{\beta,\gamma})$ to $W$.
\end{enumerate}
\end{theorem}

The next result gives the existence of a solution to the Equation \eqref{eq:EDSbis}.

\begin{theorem}
Let $\beta>1/2$, $\gamma\in[0,1)$ such that
\begin{equation*}
\beta(1+\gamma)>1.
\end{equation*}
We consider $A\in\mathcal{C}^{\beta,\gamma}([0,T];\mathbb{R}^d)$. There exists a solution $Y\in\mathcal{C}^{\beta}([0,T];\mathbb{R}^d)$ to the nonlinear Young differential equation \eqref{eq:EDSbis}. Furthermore, there exists a constant $C$ depending on $\beta,\gamma,T$ and $\|A\|_{\beta,\gamma}$ such that
\begin{equation*}
\|Y\|_{\mathcal{C}^{\beta}([0,T])}\leq C(|Y_0|+1).
\end{equation*}
\end{theorem}

We finally state a uniqueness result which only relies on the regularity of $A$.

\begin{theorem}\label{thm:CauchyCatGub}
Let $\beta>1/2$, $\gamma\in[0,1]$ such that $A\in\mathcal{C}^{\beta,\gamma+1}$. Then, there exists a unique solution  $Y\in\mathcal{C}^{\beta}([0,T];\mathbb{R}^d)$ to the nonlinear Young differential equation \eqref{eq:EDSbis}.
\end{theorem}

\subsection{Proof of Theorem \ref{thm:CauchySDE}}

To obtain such results in our context, we need Theorem \ref{thm:main} and, from there, we essentially have to derive the proper bounds on $A$ in adequate Sobolev spaces. Before proceeding in this direction, we recall the smoothing properties of the heat semigroup.
\begin{lemma}
\label{lemma:estsemigroup}
Let $m,\gamma \in\mathbb{R}$ and $p\in (1,\infty)$. For any $f\in W^{m,p}(\mathbb{R}^d)$ and $\tau\in\mathbb{R}^{+*}$, we have
\begin{equation*}
\left\|P_{\tau} f \right\|_{W^{m,p}(\mathbb{R}^d)} \lesssim \tau^{-\gamma/2}\|f\|_{W^{m-\gamma,p}(\mathbb{R}^d)}.
\end{equation*}
\end{lemma}

We are now in position to prove the following result.

\begin{prop} Under Assumption \ref{asm:chainrule}, there exists $\gamma>0$ and $\beta>1/2$ such that, up to a modification, $A\in\mathcal{C}^{\beta}([0,T];\mathcal{C}_b^{1+\gamma}(\mathbb{R}^d))$ where $\mathcal{C}_b^{1+\gamma}(\mathbb{R}^d)$ is the space of bounded and $1+\gamma$-H\"older functions.
\end{prop}

\begin{proof} \noindent\textbf{Step 1:} By Assumption \ref{asm:chainrule}, there exist $\varepsilon_1,\varepsilon_2>0$ such that
\begin{equation*}
\frac1{2H} - \frac{2}{H\ell}-\frac1{Hq}+ k-1-\frac dp = \varepsilon_1 + \frac{\varepsilon_2}H.
\end{equation*}
By Theorem \ref{thm:main} and \eqref{eq:defA}, we have that, for any $x\in\mathbb{R}^d$, $\delta A_{s,t}(x)$ is given by
\begin{align*}
\delta A_{s,t}(x) =&  \int_s^t P_{\frac{1}{2H}r^{2H}} b(r,W^{H}(r)+x)dr\nonumber
\\ & +\sum_{j=1}^d \int_s^t \int_u^t P_{\frac{1}{2H}(r-u)^{2H}} \frac{\partial}{\partial x_j} b^a(r,u,W^{2,H}(u,r)+x)(r-u)^{H-1/2}dr dB_j(u)\nonumber
\\ & +\sum_{j=1}^d \int_s^t  \int_{u}^t P_{\frac{1}{2H}(r-u)^{2H}} \frac{\partial}{\partial x_j}  \Dp b_j(r,u,W^{2,H}(u,r)+x)(r-u)^{H-1/2}drdu\nonumber
\\& - \sum_{j=1}^d \int_s^t  \int_{u}^tP_{\frac{1}{2H}(r-u)^{2H}} \Dp b_j(r,u,W^{2,H}(u,r)+x)(r-u)^{H-1/2}dr dB_j(u),
\end{align*}
where we denote
\begin{equation*}
\Dp b_j(r,u,x) = (\Dp b(r,x))_j(u).
\end{equation*}
We first estimate each term from the right-hand-side in the $L^{\ell}(\Omega;W^{1+d/p+\varepsilon_1,p}(\mathbb{R}^d))$-norm.
We denote 
\begin{equation*}
m : = 1+\frac{d}{p}+\varepsilon_1 = k  + \frac1{H}\left(\frac1{2}-\frac2{\ell}-\frac1{q} - \varepsilon_2\right).
\end{equation*}
By a density argument, we can assume that $b$ is a smooth random field. For the first term, we have, thanks to Hölder's inequality and Lemma \ref{lemma:estsemigroup},
\begin{align*}
\left\| \int_s^t P_{\frac{1}{2H}r^{2H}}b(r,\cdot)dr \right\|_{W^{m,p}(\mathbb{R}^d)} &\leq \int_s^t \left\|P_{\frac{1}{2H}r^{2H}} b(r,\cdot)\right\|_{W^{m,p}(\mathbb{R}^d)}dr \lesssim \int_s^t r^{-1/2+2/\ell+1/q+\varepsilon_2} \left\| b(r,\cdot)\right\|_{W^{k,p}(\mathbb{R}^d)}dr
\\ &\lesssim (t-s)^{1/2 +2/\ell + \varepsilon_2} \|b\|_{L^{q}([0,T];W^{k,p}(\mathbb{R}^d))}.
\end{align*}
We now turn to the second term and use the BDG inequality\footnote{Burkholder-Gavis-Gundy inequality} together with Lemma \ref{lemma:estsemigroup}, to deduce that, for any $j\in\{1,\ldots,d\}$,
\begin{align*}
\E\left[ \left\|\int_s^t \int_u^tP_{\frac{1}{2H}(r-u)^{2H}}\partial_{x_j} b^a(r,u,W^{2,H}(u,r)+\cdot)(r-u)^{H-1/2}drdB_j(u)\right\|^\ell_{W^{m,p}(\mathbb{R}^d)} \right]^{1/\ell}
\\ \lesssim \E\left[ \left(\int_s^t \left(\int_u^t (r-u)^{- 1 +2/\ell + 1/q + \varepsilon_2} \|b^a(r,u,\cdot)\|_{W^{k,p}(\mathbb{R}^d)}dr\right)^2 du\right)^{\ell/2} \right]^{1/\ell} 
\\ \lesssim \left(\int_s^t (t-u)^{2/\ell+\varepsilon_2}  \|b\|_{L^\ell(\Omega;L^{q}([0,T];W^{k,p}(\mathbb{R}^d)))}^2 du\right)^{1/2} 
\\ \lesssim  (t-s)^{(1+\varepsilon_2)/2+1/\ell} \|b\|_{L^\ell(\Omega;L^{q}([0,T];W^{k,p}(\mathbb{R}^d)))}.
\end{align*}
By similar arguments, Jensen's inequality and ii) of Assumption \ref{asm:chainrule}, we can bound the fourth term. We obtain, for any $j\in\{1,\ldots,d\}$,
\begin{align*}
\E\left[ \left\|\int_s^t \int_u^tP_{\frac{1}{2H}(r-u)^{2H}} \Dp b_j(u,r,W^{2,H}(u,r)+\cdot)(r-u)^{H-1/2}drdB_j(u)\right\|^\ell_{W^{m,p}(\mathbb{R}^d)} \right]^{1/\ell}
\\ =\E\left[ \left\|\int_s^t \int_u^tP_{\frac{1}{2H}(r-u)^{2H}} \mathbb{E}_{u}[b_j'(r,W^{2,H}(u,r)+\cdot)\upsilon(u,r)] (r-u)^{H-1/2}drdB_j(u)\right\|^\ell_{W^{m,p}(\mathbb{R}^d)} \right]^{1/\ell}
\\ \lesssim \E\left[ \left(\int_s^t \left(\int_u^t \mathbb{E}_{u}\left[\left\| P_{\frac{1}{2H}(r-u)^{2H}} b_j'(r,W^{2,H}(u,r)+\cdot)\upsilon(u,r)\right\|_{W^{m,p}(\mathbb{R}^d)}\right] (r-u)^{H-1/2}dr\right)^2du\right)^{2/\ell} \right]^{1/\ell}
\\ \lesssim \E\left[ \left(\int_s^t \left(\int_u^t (r-u)^{-1+2/\ell+1/q+\varepsilon_2} \mathbb{E}_{u}\left[C_1\|b'(r,\cdot)\|_{W^{k-\iota,p}(\mathbb{R}^d)}\right]dr\right)^2 du\right)^{\ell/2} \right]^{1/\ell} 
\\ \lesssim   (t-s)^{(1+\varepsilon_2)/2+1/\ell}  \|b'\|_{L^{\sigma}(\Omega;L^{q}([0,T];W^{k-\iota,p}(\mathbb{R}^d)))}.
\end{align*}
We finally estimate the third term. We have, for any $j\in\{1,\ldots,d\}$,
\begin{align*}
\left\|\int_s^t  \int_{u}^tP_{\frac{1}{2H}(r-u)^{2H}} \partial_{x_j} \Dp b_j (r,W^{2,H}(u,r)+\cdot)(r-u)^{H-1/2}drdu\right\|_{W^{m,p}(\mathbb{R}^d)}
\\ \lesssim \int_s^t \int_{u}^t (r-u)^{-1 +2/\ell + 1/q +\varepsilon_2}\E_u\left[ C_1 \left\| b' (r,\cdot)\right\|_{W^{k-\iota,p}(\mathbb{R}^d)}\right]drdu
\\ \lesssim \int_s^t (t-u)^{2/\ell +\varepsilon_2}\E_u\left[ C_1^2\right]^{1/2} \E_u\left[\left\| b'\right\|_{L^{q}([0,T];W^{k-\iota,p}(\mathbb{R}^d))}^2\right]^{1/2}du,
\end{align*}
which leads to
\begin{align*}
\E\left[ \left\|\int_s^t  \int_{u}^tP_{\frac{1}{2H}(r-u)^{2H}} \partial_{x_j} \Dp b_j(r,u,W^{2,H}(u,r)+\cdot)(r-u)^{H-1/2}drdu\right\|_{W^{k,p}(\mathbb{R}^d)}^{\ell}\right]^{1/\ell}
\\ \lesssim  (t-s)^{1+\varepsilon_2+2/\ell} \left\| b'\right\|_{L^{\sigma}(\Omega;L^{q}([0,T];W^{k-\iota,p}(\mathbb{R}^d)))}.
\end{align*}
\noindent\textbf{Step 2:} From the Sobolev embedding
\begin{equation*}
W^{1+d/p+\varepsilon_1,p}(\mathbb{R}^d) \hookrightarrow\mathcal{C}_b^{1+\gamma}(\mathbb{R}^d),
\end{equation*}
for any $0<\gamma<\varepsilon_1$, we deduce that
\begin{equation*}
\mathbb{E}\left[ \|\delta A_{s,t}\|_{\mathcal{C}^{1+\gamma}(\mathbb{R}^d)}^\ell\right]^{1/\ell} \leq C|t-s|^{\beta+1/\ell}\left( \|b\|_{L^{\ell}(\Omega;L^q([0,T];W^{k,p}(\mathbb{R}^d)))} + \|b'\|_{L^{\sigma}(\Omega;L^q([0,T];W^{k-\iota,p}(\mathbb{R}^d)))} \right).
\end{equation*}
It follows from Kolmogorov's continuity theorem that, up to a modification,
\begin{equation*}
A\in\mathcal{C}^{\beta}([0,T];\mathcal{C}_b^{1+\gamma}(\mathbb{R}^d)).
\end{equation*}
\end{proof}

As a direct consequence from the previous proposition, it follows from Theorem \ref{thm:CauchyCatGub}, that Equation \eqref{eq:EDSbis} admits a unique solution.

\section{Proof of Theorem \ref{thm:main}}
\label{section:proofmain}
As the reader will realise, Formula (\ref{eq:mainformula}) is valid for any fixed $x$ in $\real^d$ and any pair $(s,t)$ with $0\leq s\leq t\leq T$. Hence, to avoid cumbersome notations we fix in this proof : 
$$x=0, \quad s=0, \quad t=T.$$  
Throughout this proof, $C$ will denote a generic constant that may vary from line to line. The proof is divided into several steps. For any $N$ in $\mathbb{N}^\ast$ and $i$ in $\{0,\cdots, N\}$, we set $t_i^N:=i \frac{T}{N}$. To prevent notations to become cumbersome we will often write $t_i$ instead of $t_i^N$. \\\\
\noindent
In the following we make use of the following notation : For $i$ in $\{0,\ldots,N-1\}$, and $s\geq t_{i+1}$, we set
\begin{equation}
\label{eq:incr}
\left\lbrace
\begin{array}{l}
\delta_{i,s}(W^{2,H}):=\left(\delta_{1,i,s}(W^{2,H}),\cdots,\delta_{d,i,s}(W^{2,H})\right),\\\\
\delta_{k,i,s}(W^{2,H}):=\left(W^{2,H}(t_{i+1},s) -  W^{2,H}(t_i,s)\right)_{k}, \quad k\in \{1,\ldots,d\}.\\
\end{array}
\right.
\end{equation}

\textbf{Step 1 :}
We first assume that $f$ belongs to $\mathcal{S}_{ad}$, that is there exist
$$ n \in \mathbb{N}^*, \; 0\leq \gamma_1<\gamma_2<\cdots<\gamma_n\leq T, \quad \varphi:[0,T]\times (\real^d)^n\times\real^d \to \real $$
such that 
\begin{equation}
\label{eq:defif}
f(t,y)=\varphi(t,B(\gamma_1\wedge t),\cdots,B(\gamma_n\wedge t),y), \; y\in \real^d, \quad t\in [0,T],
\end{equation}
and $\varphi(t\cdot)$ is bounded and admits bounded partial derivatives of any order which are uniformly bounded in $t$ on $[0,T]$. Hence, for any $0\leq r\leq u\leq s\leq T$, for any $\mathcal{F}_u$-measurable random variable $G$, and for any operator $\mathcal{L}_y$ of the form $\mathcal{L}_y:=\frac{\partial^k \varphi}{\prod_{i=1}^\ell \partial_{y_i}^{k_i}}$ (with $\ell \in \mathbb{N}^*$, $v\in \{0,1,\ldots,4\}$, $v_1,\cdots,v_\ell, p_1,\cdots,p_\ell$ in $\mathbb{N}$ with $\sum_{i=1}^\ell v_i^{k_i} = v$)
\begin{align*}
&\sup_{0\leq u\leq s \leq T} \sum_{j=1}^d \left|\E_u[(P_{\frac{1}{2H}(s-r)^{2H}}  D_j \mathcal{L}_y  f(s,y))(u)]_{x=G}\right|+\sup_{0\leq u\leq s \leq T} \sum_{j=1}^d \left|\E_u[(P_{\frac{1}{2H}(s-r)^{2H}} \mathcal{L}_y  f(s,y))(u)]_{y=G}\right| \\
&\leq \sum_{i=1}^n \sup_{0\leq u\leq s \leq T} \left|\frac{\partial}{\partial_{b_i}}\mathcal{L}_y  \varphi(s,b,y))\right|<\infty.
\end{align*}
Throughout this step, $C$ will denote a generic constant which may differ from line to line and which depends on : $T$, $H$, $d$ and on : 
$$ \sup_{0\leq s \leq T} \sum_{i=1}^n \left\|\frac{\partial}{\partial_{y_i}}\mathcal{L}_y\varphi(s,\cdot)\right\|_\infty<+\infty,$$
where $\mathcal {L}_y$ denotes any partial derivative of order less or equal to $4$.    
\\\\
First of all, the Clark-Ocone formula (\ref{eq:CO}) applies to the random variable $F(t)$ (defined as (\ref{eq:definitionF})) allows one to decompose for any time $t$ the random variable $F(t)$ as follows :
\begin{equation}
\label{eq:decomp1}
F(t) = \E_t\left[ F(t) \right] + \sum_{j=1}^d \int_t^T \E_u\left[ (D_jF(t))(u)\right]dB_j(u), \quad t\in [0,T].
\end{equation}  
By defintion, $F^a(t)=\E_t\left[ F(t) \right]$ and set $G(t):=\sum_{j=1}^d \int_t^T \E_u\left[ (D_j F(t))(u)\right]dB_{j}(u)$ so that 
\begin{equation}
\label{eq:decomp1}
F(t) = F^a(t) + G(t), \quad t\in [0,T].
\end{equation} 
Using Definition (\ref{eq:definitionF}) of $F$ we have for any $i$ in $\{0,\cdots,N-1\}$ that :
\begin{align}
\label{eq:Fincr1}
F(t_{i+1}) - F(t_i) &= \int_{t_{i+1}}^T P_{\frac{1}{2H}(s-t_{i+1})^{2H}}f(s,W^{2,H}(t_{i+1},s))ds - \int_{t_i}^T P_{\frac{1}{2H}(s-t_i)^{2H}}f(s,W^{2,H}(t_i,s))ds \nonumber
\\&= - \int_{t_i}^{t_{i+1}} P_{\frac{1}{2H}(s-t_i)^{2H}}f(s,W^{2,H}(t_i,s))ds \nonumber
\\ &\hspace{2em}+  \int_{t_{i+1}}^T\left[ P_{\frac{1}{2H}(s-t_{i+1})^{2H}}f(s,W^{2,H}(t_{i+1},s)) - P_{\frac{1}{2H}(s-t_i)^{2H}}f(s,W^{2,H}(t_i,s))\right]ds \nonumber
\\&= - \int_{t_i}^{t_{i+1}} P_{\frac{1}{2H}(s-t_i)^{2H}}f(s,W^{2,H}(t_i,s))ds \nonumber
\\ &\hspace{2em}+  \int_{t_{i+1}}^T\left[P_{\frac{1}{2H}(s-t_{i+1})^{2H}} - P_{\frac{1}{2H}(s-t_i)^{2H}}\right]f(s,W^{2,H}(t_{i+1},s))ds \nonumber
\\ &\hspace{2em}+  \int_{t_{i+1}}^T P_{\frac{1}{2H}(s-t_i)^{2H}}\left[f(s,W^{2,H}(t_{i+1},s)) - f(s,W^{2,H}(t_i,s)) \right]ds.
\end{align}
We aim here to use a Taylor expansion. To this end we set (using Notation (\ref{eq:incr})) : 
\begin{equation}
\label{eq:inctheta}
W^{2,H}(t_i,s,\theta) := W^{2,H}(t_i,s) +\theta \; \delta_{i,s}(W^{2,H}), \quad \theta \in [0,1].
\end{equation}
With this notation at hand, the last term in this expression writes as follows :
\begin{align*}
&P_{\frac{1}{2H}(s-t_i)^{2H}} f(s,W^{2,H}(t_{i+1},s)) - P_{\frac{1}{2H}(s-t_i)^{2H}} f(s,W^{2,H}(t_i,s)) \\
&= \nabla P_{\frac{1}{2H}(s-t_i)^{2H}} f(s,W^{2,H}(t_i,s)) \cdot \delta_{i,s}(W^{2,H})\\
&+ \frac 1 2 \sum_{k=1}^d \frac{\partial^2}{\partial x_k^2} P_{\frac{1}{2H}(s-t_i)^{2H}} f(s,W^{2,H}(t_i,s)) \left(\delta_{i,k,s}(W^{2,H})\right)^2\\
&+ \frac 1 2 \sum_{k,\ell=1; k\neq \ell}^d \frac{\partial^2}{\partial x_k \partial x_\ell} P_{\frac{1}{2H}(s-t_i)^{2H}} f(s,W^{2,H}(t_i,s)) \delta_{i,k,s}(W^{2,H}) \delta_{i,\ell,s}(W^{2,H})\\
&+ \frac 1 6 \int_{0}^1\nabla^3 P_{\frac{1}{2H}(s-t_i)^{2H}} f\left(s,W^{2,H}(t_i,s,\theta)\right) d\theta \cdot \left(\delta_{i,s}(W^{2,H})\right)^3.
\end{align*}
To proceed with our analysis we apply the Clark-Ocone formula (\ref{eq:CO}) to each element $$\mathcal{L} P_{\frac{1}{2H}(s-t_i)^{2H}} f(s,W^{2,H}(t_i,s))$$ with $\mathcal{L}=\frac{\partial}{\partial y_k}$ (for $k$ in $\{1,\cdots,d\}$) or $\mathcal{L}=\frac{\partial^2}{\partial y_k \partial_{y_\ell}}$ for $k,\ell$ in $\{1,\cdots,d\}$ with $k\neq \ell$. We have 
\begin{align*}
&\mathcal{L} P_{\frac{1}{2H}(s-t_i)^{2H}} f(s,W^{2,H}(t_i,s)) \\
&= \E_{t_i}\left[\mathcal{L} P_{\frac{1}{2H}(s-t_i)^{2H}} f(s,W^{2,H}(t_i,s))\right] + \sum_{j=1}^d \int_{t_i}^s \E_u\left[D_j \left(\mathcal{L} P_{\frac{1}{2H}(s-t_i)^{2H}}f(s,W^{2,H}(t_i,s))\right)(u)\right] dB_j(u).
\end{align*}
Since $W^{2,H}(t_i,s)$ is $\mathcal{F}_{t_i}$-measurable, the first term of the right hand side is : 
$$\E_{t_i}\left[\mathcal{L} P_{\frac{1}{2H}(s-t_i)^{2H}} f(s,W^{2,H}(t_i,s))\right] = \mathcal{L} P_{\frac{1}{2H}(s-t_i)^{2H}} f^a(s,t_i,W^{2,H}(t_i,s)),$$ 
whereas Lemma \ref{lemma:propMalliavin} implies that : 
$$ \left(\E_u\left[D_j \left(\mathcal{L} P_{\frac{1}{2H}(s-t_i)^{2H}}f(s,W^{2,H}(t_i,s))\right)(u)\right]\right)_{t_i\leq u\leq s} = \left(\mathcal{L} P_{\frac{1}{2H}(s-t_i)^{2H}}g_j(s,u,W^{2,H}(t_i,s))\right)_{t_i\leq u\leq s},$$
where the equality is understood as processes in $L^2(\Omega\times [0,T])$ and where we recall Notation (\ref{eq:notationproof}). Hence 
\begin{align*}
&\mathcal{L} P_{\frac{1}{2H}(s-t_i)^{2H}} f(s,W^{2,H}(t_i,s)) \\
&= \mathcal{L} P_{\frac{1}{2H}(s-t_i)^{2H}} f^a(s,t_i,W^{2,H}(t_i,s)) + \sum_{j=1}^d \int_{t_i}^s \mathcal{L} P_{\frac{1}{2H}(s-t_i)^{2H}}g_j(s,u,W^{2,H}(t_i,s)) dB_j(u).
\end{align*}

Thus
\begin{align*}
&P_{\frac{1}{2H}(s-t_i)^{2H}} f(s,W^{2,H}(t_{i+1},s)) - P_{\frac{1}{2H}(s-t_i)^{2H}} f(s,W^{2,H}(t_i,s)) \\
&= \sum_{k=1}^d \frac{\partial}{\partial x_k} P_{\frac{1}{2H}(s-t_i)^{2H}} f^a(s,t_i,W^{2,H}(t_i,s)) \; \delta_{k,i,s}(W^{2,H})
\\ &\hspace{1em} + \sum_{k=1}^d \sum_{j=1}^d \int_{t_i}^{t_{i+1}} \frac{\partial}{\partial x_k} P_{\frac{1}{2H}(s-t_i)^{2H}} g_j(s,u,W^{2,H}(t_i,s)) dB_j(u) \; \delta_{k,i,s}(W^{2,H})
\\ &\hspace{1em} + \sum_{k=1}^d \sum_{j=1}^d \int_{t_{i+1}}^{s} \frac{\partial}{\partial x_k} P_{\frac{1}{2H}(s-t_i)^{2H}} g_j(s,u,W^{2,H}(t_i,s)) dB_j(u) \; \delta_{k,i,s}(W^{2,H})
\\ &\hspace{1em} + \frac12  \sum_{k=1}^d \frac{\partial^2}{\partial x_k^2} P_{\frac{1}{2H}(s-t_i)^{2H}} f(s,W^{2,H}(t_i,s)) dB_j(u) \; \left(\delta_{k,i,s}(W^{2,H})\right)^2 
\\ &\hspace{1em} + \frac12 \sum_{k,\ell=1;k\neq \ell}^d\frac{\partial^2}{\partial x_k \partial x_\ell} P_{\frac{1}{2H}(s-t_i)^{2H}} f^a(s,t_i,W^{2,H}(t_i,s)) dB_j(u) \; \delta_{k,i,s}(W^{2,H}) \delta_{\ell,i,s}(W^{2,H})
\\ &\hspace{1em} + \frac12  \sum_{k,\ell=1;k\neq \ell}^d \sum_{j=1}^d \int_{t_i}^{t_{i+1}}\frac{\partial^2}{\partial x_k \partial x_\ell} P_{\frac{1}{2H}(s-t_i)^{2H}} g_j(s,u,W^{2,H}(t_i,s)) dB_j(u) \; \delta_{k,i,s}(W^{2,H}) \delta_{\ell,i,s}(W^{2,H})
\\ &\hspace{1em} + \frac12 \sum_{k,\ell=1;k\neq \ell}^d \sum_{j=1}^d \int_{t_{i+1}}^{s}\frac{\partial^2}{\partial x_k \partial x_\ell} P_{\frac{1}{2H}(s-t_i)^{2H}} g_j(s,u,W^{2,H}(t_i,s)) dB_j(u) \; \delta_{k,i,s}(W^{2,H}) \delta_{\ell,i,s}(W^{2,H})
\\ &\hspace{1em}+ \frac 1 6 \int_0^1 \nabla^3 P_{\frac{1}{2H}(s-t_i)^{2H}} f\left(s,W^{2,H}(t_i,s,\theta)\right) d\theta \cdot \left(\delta_{i,s}(W^{2,H})\right)^3.
\end{align*}
Coming back to the expression (\ref{eq:Fincr1}) of an increment of $F$ we obtain 
\begin{align}
\label{eq:decF2}
& F(t_{i+1}) - F(t_i) \nonumber\\
&= - \int_{t_i}^{t_{i+1}} P_{\frac{1}{2H}(s-t_i)^{2H}}f(s,W^{2,H}(t_i,s))ds\nonumber
\\ &\hspace{2em}+  \int_{t_{i+1}}^T\left[P_{\frac{1}{2H}(s-t_{i+1})^{2H}} - P_{\frac{1}{2H}(s-t_i)^{2H}}\right]f(s,W^{2,H}(t_{i+1},s))ds\nonumber
\\ &\hspace{2em}+ \sum_{k=1}^d \int_{t_{i+1}}^T \frac{\partial}{\partial x_k} P_{\frac{1}{2H}(s-t_i)^{2H}} f^a(s,t_i,W^{2,H}(t_i,s)) \delta_{k,i,s}(W^{2,H}) ds\nonumber
\\ &\hspace{2em} + \sum_{k=1}^d \sum_{j=1}^d \int_{t_{i+1}}^T \int_{t_i}^{t_{i+1}} \frac{\partial}{\partial x_k} P_{\frac{1}{2H}(s-t_i)^{2H}} g_j(s,u,W^{2,H}(t_i,s)) dB_j(u) \; \delta_{k,i,s}(W^{2,H}) ds\nonumber
\\ &\hspace{2em} + \sum_{k=1}^d \sum_{j=1}^d \int_{t_{i+1}}^T \int_{t_{i+1}}^{s} \frac{\partial}{\partial x_k} P_{\frac{1}{2H}(s-t_i)^{2H}} g_j(s,u,W^{2,H}(t_i,s)) dB_j(u) \; \delta_{k,i,s}(W^{2,H}) ds\nonumber
\\ &\hspace{2em} + \frac12 \sum_{k=1}^d \int_{t_{i+1}}^T \frac{\partial^2}{\partial x_k^2} P_{\frac{1}{2H}(s-t_i)^{2H}} f(s,W^{2,H}(t_i,s)) dB_j(u) \; \left(\delta_{k,i,s}(W^{2,H})\right)^2 ds\nonumber
\\ &\hspace{2em} + \frac12 \sum_{k,\ell=1;k\neq \ell}^d \int_{t_{i+1}}^T \frac{\partial^2}{\partial x_k \partial x_\ell} P_{\frac{1}{2H}(s-t_i)^{2H}} f^a(s,t_i,W^{2,H}(t_i,s)) \; \delta_{k,i,s}(W^{2,H}) \delta_{\ell,i,s}(W^{2,H}) ds\nonumber
\\ &\hspace{2em} + \frac12 \sum_{k,\ell=1;k\neq \ell}^d \sum_{j=1}^d \int_{t_{i+1}}^T \int_{t_{i}}^{\tio}\frac{\partial^2}{\partial x_k \partial x_\ell} P_{\frac{1}{2H}(s-t_i)^{2H}} g_j(s,u,W^{2,H}(t_i,s)) dB_j(u) \; \delta_{k,i,s}(W^{2,H}) \delta_{\ell,i,s}(W^{2,H}) ds\nonumber
\\ &\hspace{2em} + \frac12 \sum_{k,\ell=1;k\neq \ell}^d \sum_{j=1}^d \int_{t_{i+1}}^T \int_{t_{i+1}}^{s}\frac{\partial^2}{\partial x_k \partial x_\ell} P_{\frac{1}{2H}(s-t_i)^{2H}} g_j(s,u,W^{2,H}(t_i,s)) dB_j(u) \; \delta_{k,i,s}(W^{2,H}) \delta_{\ell,i,s}(W^{2,H}) ds\nonumber
\\ &\hspace{2em}+  \frac16 \int_{t_{i+1}}^T \int_0^1 \nabla^3 P_{\frac{1}{2H}(s-t_i)^{2H}} f\left(s,W^{2,H}(t_i,s,\theta)\right) d\theta \cdot \left(\delta_{i,s}(W^{2,H})\right)^3 ds\nonumber
\\ & =: \sum_{k = 1}^{10}  I_{1,k}(t_i,t_{i+1}).
\end{align}
We now compute an increment of $G$. To this end we first remark that (recall Notation in (\ref{eq:CO}))
\begin{align}
G(t) &: = \sum_{j=1}^d \int_t^T (\Dp_j F(t))(u) dB_{j}(u) \nonumber\\
&= \sum_{j=1}^d \int_t^T \E_u\left[ D_j \left(\int_t^T P_{\frac{1}{2H}(s-t)^{2H}}f(s,W^{2,H}(t,s))ds\right)(u) \right]dB_{j}(u)\nonumber\\ 
&= \sum_{j=1}^d \int_t^T \int_t^T \E_u\left[D_j \left(P_{\frac{1}{2H}(s-t)^{2H}} f(s,W^{2,H}(t,s))\right)(u) \bold{1}_{\{u\leq s\}}\right]ds dB_j(u) \nonumber\\ 
&= \sum_{j=1}^d \int_t^T \int_u^T \Dp_j \left(P_{\frac{1}{2H}(s-t)^{2H}} f(s,W^{2,H}(t,s))\right)(u) ds dB_j(u),
\end{align}
where the first equality is a consequence of the stochastic Fubini theorem as for any $j$ in $\{1,\ldots,d\}$
\begin{align*}
&\int_t^T \E\left[\left|\E_u\left[ D_j\left(\int_{u}^T P_{\frac{1}{2H}(s-t)^{2H}}f(s,W^{2,H}(t,s))ds\right)(u) \right]\right|^2\right]du\\
&\int_t^T \int_{u}^T \E\left[\left|D_j\left(P_{\frac{1}{2H}(s-t)^{2H}}f(s,W^{2,H}(t,s))\right)(u)\right|^2\right]du ds\\
&\leq C \int_t^T \left|T-u\right|^2du <+\infty.
\end{align*}
In addition, since for any $t$, $W^{2,H}(t,s)$ is $\mathcal{F}_t$-measurable, Lemma \ref{lemma:propMalliavin} implies that 
$$ \left(\Dp \left(P_{\frac{1}{2H}(s-t)^{2H}} f(s,W^{2,H}(t,s))\right)(u) \right)_u = \left(P_{\frac{1}{2H}(s-t)^{2H}} g(s,u,W^{2,H}(t,s))\right)_u.$$
Thus,
$$ G(t) = \sum_{j=1}^d \int_t^T \int_u^T P_{\frac{1}{2H}(s-t)^{2H}} g(s,u,W^{2,H}(t,s)) ds dB_j(u).$$
This form allows us to proceed in the analysis of an increment of $G$. Indeed,
\begin{align*}
&G(t_{i+1})-G(t_i) \nonumber\\
&= \sumj \int_{t_{i+1}}^T \int_u^T \left[ P_{\frac{1}{2H}(s-t_{i+1})^{2H}} g_j(s,u,W^{2,H}(t_{i+1},s)) - P_{\frac{1}{2H}(s-t_{i})^{2H}} g_j(s,u,W^{2,H}(t_{i},s)) \right]ds dB_j(u)
\\ &\hspace{2em} - \sumj \int_{t_i}^{t_{i+1}} \int_u^T P_{\frac{1}{2H}(s-t_{i})^{2H}} g_j(s,u,W^{2,H}(t_{i},s)) ds dB_j(u)\\
&= \sumj \int_{t_{i+1}}^T \int_u^T \left[ P_{\frac{1}{2H}(s-t_{i+1})^{2H}} g_j(s,u,W^{2,H}(t_{i+1},s)) - P_{\frac{1}{2H}(s-t_{i})^{2H}} g_j(s,u,W^{2,H}(t_{i+1},s)) \right]ds dB_j(u)\\
&+ \sumj \int_{t_{i+1}}^T \int_u^T \left[ P_{\frac{1}{2H}(s-t_{i})^{2H}} g_j(s,u,W^{2,H}(t_{i+1},s)) - P_{\frac{1}{2H}(s-t_{i})^{2H}} g_j(s,u,W^{2,H}(t_{i},s)) \right]ds dB_j(u)\\
&\hspace{2em} - \sumj \int_{t_i}^{t_{i+1}} \int_u^T P_{\frac{1}{2H}(s-t_{i})^{2H}} g_j(s,u,W^{2,H}(t_{i},s)) ds dB_j(u)
\end{align*}
In a similar fashion than the computation of an increment of $F$, we expand using Taylor expansion the second term to obtain
\begin{align*}
&P_{\frac{1}{2H}(s-t_{i})^{2H}} g_j(s,u,W^{2,H}(t_{i+1},s)) - P_{\frac{1}{2H}(s-t_{i})^{2H}} g_j(s,u,W^{2,H}(t_{i},s))\\
&=\nabla P_{\frac{1}{2H}(s-t_{i})^{2H}} g_j(s,u,W^{2,H}(t_{i},s)) \cdot \delta_{i,s}(W^{2,H}) + \frac12 \nabla^2 P_{\frac{1}{2H}(s-t_{i})^{2H}} g_j(s,u,W^{2,H}(t_{i},s)) \cdot \left(\delta_{i,s}(W^{2,H})\right)^2\\
&\hspace{2em}+ \frac16 \int_0^1 \nabla^3 P_{\frac{1}{2H}(s-t_{i})^{2H}} g_j(s,u,W^{2,H}(t_i,s,\theta)) d\theta \cdot \left(\delta_{i,s}(W^{2,H})\right)^3,
\end{align*}
where we recall Notation (\ref{eq:inctheta}). Plugging this expansion in the expression above, we get 
\begin{align}
\label{eq:decG1}
&G(t_{i+1})-G(t_i) \nonumber\\
&= \sumj \int_{t_{i+1}}^T \int_u^T \left[P_{\frac{1}{2H}(s-t_{i+1})^{2H}} g_j(s,u,W^{2,H}(t_{i+1},s)) - P_{\frac{1}{2H}(s-t_{i})^{2H}} g_j(s,u,W^{2,H}(t_{i+1},s)) \right]ds dB_j(u) \nonumber\\
&+ \sumj \int_{t_{i+1}}^T \int_u^T \left[ \nabla P_{\frac{1}{2H}(s-t_{i})^{2H}} g_j(s,u,W^{2,H}(t_{i},s)) \cdot \delta_{i,s}(W^{2,H})) \right]ds dB_j(u)\nonumber\\
&+ \frac12 \sumj \int_{t_{i+1}}^T \int_u^T \left[ \nabla^2 P_{\frac{1}{2H}(s-t_{i})^{2H}} g_j(s,u,W^{2,H}(t_{i},s)) \cdot \left(\delta_{i,s}(W^{2,H})\right)^2) \right]ds dB_j(u)\nonumber\\
&+ \frac16 \sumj \int_{t_{i+1}}^T \int_u^T \left[ \int_0^1 \nabla^3 P_{\frac{1}{2H}(s-t_{i})^{2H}} g_j(s,u,W^{2,H}(t_i,s,\theta)) d\theta \cdot \left(\delta_{i,s}(W^{2,H})\right)^3) \right]ds dB_j(u)\nonumber\\
&\hspace{2em} - \sumj \int_{t_i}^{t_{i+1}} \int_u^T P_{\frac{1}{2H}(s-t_{i})^{2H}} g_j(s,u,W^{2,H}(t_{i},s)) ds dB_j(u) \nonumber\\ 
& =: \sum_{k = 1}^5 I_{2,k}(t_i,t_{i+1}).
\end{align}
As a consequence, using Relation (\ref{eq:decomp1}) with $t=0$, we get : 
\begin{align}
\label{eq:sumofterms}
F^a(0) &= F(0)-G(0)\nonumber\\
&= -\lim_{N\to+\infty} \sum_{i=0}^{N-1} \left(F(t_{i+1})-F(t_{i})-(G(t_{i+1})-G(t_{i}))\right) \nonumber\\
&= -\lim_{N\to+\infty} \sum_{i=0}^{N-1} \left(\sum_{k = 1}^8 I_{1,k}(t_i,t_{i+1}) - \sum_{k = 1}^5 I_{2,k}(t_i,t_{i+1})\right) \nonumber\\
&=-\lim_{N\to+\infty} \sum_{i=0}^{N-1} \left(I_{1,1}(t_i,t_{i+1})+ I_{1,3}(t_i,t_{i+1}) + I_{1,4}(t_i,t_{i+1})(t_i,t_{i+1}) - I_{2,5}(t_i,t_{i+1}) \right) \\
&+ \lim_{N\to+\infty} \sum_{i=0}^{N-1} R(t_i,t_{i+1})\nonumber,
\end{align}
with $$R(t_i,t_{i+1}):= I_{1,2}(t_i,t_{i+1}) + I_{1,5}(t_i,t_{i+1}) + \sum_{k=6}^{10} I_{1,k}(t_i,t_{i+1}) - \sum_{k=1}^4 I_{2,k}(t_i,t_{i+1}),$$
where the terms involved in this expression are defined in (\ref{eq:decF2}) and in (\ref{eq:decG1}).\\\\
\noindent
By Lemma \ref{lemma:tec1} (postponed at the end of this section), we have that
\begin{equation}
\label{eq:I11}
\lim_{N\to+\infty} \sum_{i=0}^{N-1} I_{1,1}(t_i,t_{i+1}) = - \int_0^T f(t,W^H_t)dt,
\end{equation}
\begin{equation}
\label{eq:I13}
\lim_{N\to+\infty} \sum_{i=0}^{N-1} I_{1,3}(t_i,t_{i+1}) = \sum_{k=1}^d \int_0^T \int_t^TP_{\frac{1}{2H}(s-t)^{2H}}\nabla f^a_{t}(s,W^{2,H}(t,s))(s-t)^{H-1/2}ds \cdot dB(t),
\end{equation}
\begin{equation}
\label{eq:I14}
\lim_{N\to+\infty} \sum_{i=0}^{N-1} I_{1,4}(t_i,t_{i+1})(t_i,t_{i+1}) = \sum_{j=1}^d \int_0^T  \int_{t}^T \frac{\partial}{\partial x_j} P_{\frac{1}{2H}(s-t)^{2H}} g_j(s,t,W^{2,H}(t,s)) (s-t)^{H-1/2}dsdt,
\end{equation}
\begin{equation}
\label{eq:I24}
\lim_{N\to+\infty} \sum_{i=0}^{N-1} I_{2,5}(t_i,t_{i+1}) = -\int_0^T \int_u^T P_{\frac{1}{2H}(s-t)^{2H}} g_j(s,t,W^{2,H}(t,s)) ds dB_j(t),
\end{equation}
and that 
\begin{equation}
\label{eq:R}
\lim_{N\to+\infty} \sum_{i=0}^{N-1} R(t_i,t_{i+1}) = 0.
\end{equation}
\textbf{Step 2 :}\\\\
\noindent
In a first step, we have proved Formula (\ref{eq:mainformula}) for $f$ in $\mathcal{S}_{ad}$ for any $(s,t,x)$ in $[0,T]^2\times\real^d$ ($s\leq t$). We now extend it to any element $f$ in $\mathbb{D}^{1,m-\alpha,p}_p$. To this end, we set the operators : 
$$
\left\lbrace
\begin{array}{lll}
\mathcal{A}_{LHS} : &\mathbb{D}^{1,m-\alpha,p}_p &\to L^\infty([0,T];L^p(\Omega; W^{m,p}(\real^d)))\\
&f &\mapsto \left(\mathcal{A}_{LHS}(t,x)\right)_{t\in[0,T],x\in \real^d},
\end{array}
\right.
$$
$$ \mathcal{A}_{LHS}(t,x):=\int_0^t f(r,W_r^H+x) dr;$$
and
$$
\left\lbrace
\begin{array}{lll}
\mathcal{A}_{RHS} : &\mathcal{S}_{ad} &\to L^\infty([0,T];L^p(\Omega; W^{m,p}(\real^d)))\\
&f &\mapsto \left(\mathcal{A}_{RHS}(t,x)\right)_{t\in[0,T],x\in \real^d},
\end{array}
\right.
$$
with 
\begin{align}
\label{eq:mainformulabis}
\mathcal{A}_{RHS}(f)(t,x): =&  \int_0^t P_{\frac{1}{2H}r^{2H}} f(r,W^{H}(r)+x)dr\nonumber
\\ & +\sum_{j=1}^d \int_0^t \int_u^t P_{\frac{1}{2H}(r-u)^{2H}} \frac{\partial}{\partial x_j} f^a(r,u,W^{2,H}(u,r)+x)(r-u)^{H-1/2}dr dB_j(u)\nonumber
\\ & +\sum_{j=1}^d \int_0^t  \int_{u}^t P_{\frac{1}{2H}(r-u)^{2H}} \frac{\partial}{\partial x_j} g_j(r,u,W^{2,H}(u,r)+x)(r-u)^{H-1/2}drdu\nonumber
\\& - \sum_{j=1}^d \int_0^t  \int_{u}^t P_{\frac{1}{2H}(r-u)^{2H}} g_j(r,u,W^{2,H}(u,r)+x)(r-u)^{H-1/2}dr dB_j(u).
\end{align}
In Step 1, we have proved that for any $f$ in $\mathcal{S}_{ad}$

$$ \mathcal{A}_{LHS}=\mathcal{A}_{RHS}, \, \textrm{ in } L^\infty([0,T];L^p(\Omega; W^{m,p}(\real^d)). $$
Note also that by definition, 
$$ \|\mathcal{A}_{LHS}(f)\| \leq \|f\|_{L^\infty([0,T];L^p(\Omega; W^{m,p}(\real^d)))}. $$
So Formula (\ref{eq:mainformula}) holds true for any adapted random field $f$ in $\mathbb{D}^{1,m-\alpha,p}_p$ (that is the equality of the operators $\mathcal{A}_{LHS}$ and $\mathcal{A}_{RHS}$) is we prove that $\mathcal{A}_{RHS}$ is a well-defined bounded operator on $\mathbb{D}^{1,m-\alpha,p}_p$. We thus prove that for any adapted random field $f$ in $\mathbb{D}^{1,m-\alpha,p}_p$ we have that :
\begin{equation}
\label{eq:boundop}
\|\mathcal{A}_{RHS}(f)\|_{L^\infty([0,T];L^p(\Omega; W^{m,p}(\real^d)))} \lesssim \|f\|_{\mathbb{D}^{1,m-\alpha,p}_p}.
\end{equation}
\textbf{Proof of (\ref{eq:boundop}) :\\} We remark that the following estimates are different from the ones in the proof of Theorem \ref{thm:CauchySDE} (see Remark \ref{rmk:bndAssm}).
Let $f$ be an adapted random field in $\mathbb{D}^{1,m-\alpha,p}_p$. 
We now estimate each term in the $L^{\infty}([0,T];L^{p}(\Omega;W^{m,p}(\mathbb{R}^d)))$ space with $p\geq 2$ and $1/2-H\alpha-1/p>0$. For the first term, we have, by H\"older's inequality,
\begin{align*}
\left\| \int_0^t P_{\frac{1}{2H}r^{2H}}f(r,\cdot)dr \right\|_{L^{p}(\Omega;W^{m,p}(\mathbb{R}^d))} &\leq \int_0^t \left\|P_{\frac{1}{2H}r^{2H}} f(r,\cdot)\right\|_{L^{p}(\Omega;W^{m,p}(\mathbb{R}^d))}dr 
\\&\lesssim \int_0^t r^{-H\alpha} \left\| f(r,\cdot)\right\|_{L^{p}(\Omega;W^{m-\alpha,p}(\mathbb{R}^d))}dr
\\ &\lesssim t^{1-H\alpha - 1/p}  \|f\|_{L^{p}([0,T]\times\Omega;W^{m-\alpha,p}(\mathbb{R}^d)))},
\end{align*}
which yields
\begin{equation*}
\left\| \int_0^{\cdot} P_{\frac{1}{2H}r^{2H}}f(r,\cdot)dr \right\|_{L^{\infty}([0,T];L^{p}(\Omega;W^{m,p}(\mathbb{R}^d)))} \lesssim T^{1-H\alpha-1/p} \|f\|_{L^{q}([0,T]\times\Omega;W^{m-\alpha,p}(\mathbb{R}^d)))}.
\end{equation*}
We now turn to the second term. It follows from the BDG, Minkowski and H\"older inequalities that, for any $j\in\{1,\ldots,d\}$,
\begin{align*}
\left\|\int_0^t \int_u^t P_{\frac{1}{2H}(r-u)^{2H}} \frac{\partial}{\partial x_j} f^a(r,u,W^{2,H}(u,r)+x)(r-u)^{H-1/2}dr dB_j(u)\right\|_{L^{p}(\Omega;W^{m,p}(\mathbb{R}^d))}
\\ \lesssim \left( \int_0^t\left\|\int_u^t (r-u)^{-1/2-H\alpha}\mathbb{E}_u\left[\|f(r,\cdot)\|_{W^{m-\alpha,p}(\mathbb{R}^d)}\right] dr\right\|_{L^{p}(\Omega)}^2du\right)^{1/2}
\\ \lesssim \left( \int_0^t (t-u)^{1-2H\alpha-2/q}\|f\|_{L^p([0,T]\times\Omega;W^{m-\alpha,p}(\mathbb{R}^d)))}^2du\right)^{1/2}
\\ \lesssim T^{1-H\alpha-1/q}\|f\|_{L^q([0,T];L^{p}(\Omega;W^{m-\alpha,p}(\mathbb{R}^d)))}.
\end{align*}
By rather similar arguments, we estimate the fourth term as
\begin{align*}
\left\|\int_0^t \int_u^t P_{\frac{1}{2H}(r-u)^{2H}} g_j^a(r,u,W^{2,H}(u,r)+x)(r-u)^{H-1/2}dr dB_j(u)\right\|_{L^{p}(\Omega;W^{m,p}(\mathbb{R}^d))}
\\ \lesssim \left( \int_0^t\left\|\int_u^t (r-u)^{-1/2-H(\alpha-1)}\mathbb{E}_u\left[\|g_j(r,u,\cdot)\|_{W^{m-\alpha,p}(\mathbb{R}^d)}\right] dr\right\|_{L^{p}(\Omega)}^2du\right)^{1/2}
\\ \lesssim \left( \int_0^t (t-u)^{1-2H(\alpha-1)-2/p}\|g_j(,\cdot,u,\cdot)\|_{L^p([0,T]\times\Omega;W^{m-\alpha,p}(\mathbb{R}^d)))}^2du\right)^{1/2}
\\ \lesssim T^{1-H(\alpha-1)-3/(2p)}\|g_j\|_{L^p([0,T]^2\times\Omega;W^{m-\alpha,p}(\mathbb{R}^d)))}.
\end{align*}
Finally, we have, for the third term,
\begin{align*}
\left\|\int_0^t  \int_{u}^t P_{\frac{1}{2H}(r-u)^{2H}} \frac{\partial}{\partial x_j} g_j(r,u,W^{2,H}(u,r)+x)(r-u)^{H-1/2}drdu\right\|_{L^p(\Omega;W^{m,p}(\mathbb{R}^d))}
\\ \lesssim \int_0^t  \int_{u}^t (r-u)^{-1/2-H\alpha} \left\|g_j(r,u,\cdot)\right\|_{L^p(\Omega;W^{m,p}(\mathbb{R}^d))}drdu
\\ \lesssim T^{3/2 - H\alpha - 2/p}\|g_j\|_{L^p([0,T]^2\times\Omega;W^{m-\alpha,p}(\mathbb{R}^d)))}.
\end{align*}
Since each term in (\ref{eq:mainformula}) is linear with respect to $f$ and from each of the previous estimates, we can deduce that Formula (\ref{eq:mainformula}) is in force for any $f$ in $\mathbb{D}^{1,m-\alpha,p}_p$.
\hfill$\square$\\\par\medbreak

\begin{lemma}
\label{lemma:tec1}
With the notations of the proof of Theorem \ref{thm:main}, the convergences (\ref{eq:I11})-(\ref{eq:I24}) hold true in $L^2(\Omega)$:
\begin{itemize}
\item[(i)]
$$ \lim_{N\to+\infty} \sum_{i=0}^{N-1} I_{1,1}(t_i,t_{i+1}) = - \int_0^T f(t,W^H_t)dt. $$
\item[(ii)]
$$ \lim_{N\to+\infty} \sum_{i=0}^{N-1} I_{1,3}(t_i,t_{i+1}) = \int_0^T \int_t^TP_{\frac{1}{2H}(s-t)^{2H}}\nabla f^a_{t}(s,W^{2,H}(t,s))(s-t)^{H-1/2}ds \cdot dB(t). $$
\item[(iii)] 
$$ \lim_{N\to+\infty} \sum_{i=0}^{N-1} I_{1,4}(t_i,t_{i+1})(t_i,t_{i+1}) = \sum_{j=1}^d \int_0^T  \int_{t}^T \frac{\partial}{\partial x_j} P_{\frac{1}{2H}(s-t)^{2H}} g_j(s,t,W^{2,H}(t,s)) (s-t)^{H-1/2}dsdt, $$
(also see Remark \ref{rem:MallDiv} for this term).
\item[(iv)]
$$\lim_{N\to+\infty} \sum_{i=0}^{N-1} I_{2,5}(t_i,t_{i+1}) = -\sum_{j=1}^d \int_0^T \int_u^T P_{\frac{1}{2H}(s-t)^{2H}} g_j(s,t,W^{2,H}(t,s)) ds dB_j(t).$$
\end{itemize}
\end{lemma}

\begin{proof}
Throughout this proof, $C$ denotes a positive constant (which can vary from line to line) and that represents the sup norm of $f$ and its derivatives up to order $4$.
\textit{} \vspace{1em}\\ 
\textbf{Proof of (i) : \\\\} 
We set using Decomposition (\ref{eq:decompositionfBm}), $W^{2,H}(s,s):=W^{H}(s)$, for any $s$. We have that   
\begin{align*}
&I_{1,1}(t_i,t_{i+1})-\int_{t_i}^{t_{i+1}} f(s,W^{H}(s))ds\\
&=- \int_{t_i}^{t_{i+1}} \left(P_{\frac{1}{2H}(s-t_i)^{2H}}f(s,W^{2,H}(t_i,s))-f(s,W^{2,H}(s,s))ds\right)ds\\
&=- \int_{t_i}^{t_{i+1}} \left(P_{\frac{1}{2H}(s-t_i)^{2H}}f(s,W^{2,H}(t_i,s))-P_{0}f(s,W^{2,H}(t_i,s))\right)ds\\
&- \int_{t_i}^{t_{i+1}} \left(P_{0}f(s,W^{2,H}(t_i,s))-P_{0}f(s,W^{2,H}(s,s))\right)ds.
\end{align*}
Since the semigroup $P$ is associated to the heat equation, the first term of the right-hand side can be re-written as : 
\begin{align*}
&=\left|\int_{t_i}^{t_{i+1}} \left(P_{\frac{1}{2H}(s-t_i)^{2H}}f(s,W^{2,H}(t_i,s))-f(s,W^{2,H}(s,s))\right)ds\right|\\
&\leq \frac12 \int_{t_i}^{t_{i+1}} \int_{0}^{\frac{1}{2H}(s-t_i)^{2H}} |\Delta P_{r}f(s,W^{2,H}(t_i,s))|drds\\
&\leq \frac{C}{4H} \sup_{s\in [0,T]} \int_{t_i}^{t_{i+1}} (s-t_i)^{2H}ds.
\end{align*}
Thus, 
\begin{align*}
&\E\left[\left| \sum_{i=0}^{N-1} \int_{t_i}^{t_{i+1}} \left(P_{\frac{1}{2H}(s-t_i)^{2H}}f(s,W^{2,H}(t_i,s))-f(s,W^{2,H}(s,s))\right)ds\right|^2\right]\\
& \leq C \left| \sum_{i=0}^{N-1} \int_{t_i}^{t_{i+1}} (s-t_i)^{2H}ds\right|^2 \\
&\underset{N\to+\infty}{\longrightarrow} 0.
\end{align*}
We now turn to the second term. 
Since 
\begin{equation}
\label{eq:incfrac}
\E\left[\left|W^{2,H}(u,s)-W^{2,H}(v,s)\right|^2\right] \leq |u-v|^{\min\{2H,1\}}\quad \forall u, v \leq s,
\end{equation}
we deduce that  
\begin{align*}
& \E\left[\left|\sum_{i = 0}^{N-1} \int_{t_i}^{t_{i+1}} \left(f(s,W^{2,H}(t_i,s))-f(s,W^{2,H}(s,s))\right)ds\right|^2\right]^{1/2}\\
&\leq C \sum_{i = 0}^{N-1} \int_{t_i}^{t_{i+1}} \E\left[\left|W^{2,H}(t_i,s)-W^{2,H}(s,s)\right|^2\right]^{1/2}ds\\
&\leq C \sum_{i = 0}^{N-1} |\tio-t_i|^{1+\min\{H,1/2\}}\\
&\underset{N\to+\infty}{\longrightarrow} 0.
\end{align*}
So Item (i) (or equivalently (\ref{eq:I11})) is proved.\\\\ 
\noindent
\textbf{Proof of (ii) : \\\\} 
Fix $k$ in $\{1,\cdots,d\}$. First note that as $f$ belongs to $\mathcal{S}_{ad}$, and since $W^{2,H}(t_i,s)$ is $\mathcal{F}_{t_i}$-measurable $(s\geq t_{i+1})$, we have that : 
$$ \E_{t_i}\left[\frac{\partial}{\partial y_k} P_{\frac{1}{2H}(s-t_i)^{2H}} f(s,W^{2,H}(t_i,s))\right] = \frac{\partial}{\partial y_k} P_{\frac{1}{2H}(s-t_i)^{2H}} f^a_{t_i}(s,W^{2,H}(t_i,s)).$$
Hence, (ii) will be proved if the following holds true for any $k$ in $\{1,\ldots,d\}$ : 
\begin{align}
\label{eq:I13temp1}
&\lim_{N\to+\infty} \sum_{i=0}^{N-1} \int_{t_{i+1}}^T \frac{\partial}{\partial y_k} P_{\frac{1}{2H}(s-t_i)^{2H}} f^a_{t_i}(s,W^{2,H}(t_i,s)) \delta_{k,i,s}(W^{2,H}) ds \nonumber \\
&\overset{L^2(\Omega)}{=} \int_0^T \int_t^T \frac{\partial}{\partial y_k} P_{\frac{1}{2H}(s-t)^{2H}} f^a_{t}(s,W^{2,H}(t,s))(s-t)^{H-1/2}ds dB_k(t).
\end{align}
By definition, (recall Definition (\ref{eq:incr}) for the increments of $W^{2,H}$)
\begin{align*}
I_{1,3,k}(\ti,\tio)&:= \int_{t_{i+1}}^T \frac{\partial}{\partial y_k} P_{\frac{1}{2H}(s-t_i)^{2H}} f^a_{t_i}(s,W^{2,H}(t_i,s)) \delta_{k,i,s}(W^{2,H}) ds \\
&= \int_{t_{i+1}}^T \frac{\partial}{\partial y_k} P_{\frac{1}{2H}(s-t_i)^{2H}} f^a_{t_i}(s,W^{2,H}(t_i,s)) \int_{t_i}^{t_{i+1}} (s-u)^{H-1/2} dB_k(u) ds\\
&= \int_{t_i}^{t_{i+1}} \int_{t_{i+1}}^T \frac{\partial}{\partial y_k} P_{\frac{1}{2H}(s-t_i)^{2H}} f^a_{t_i}(s,W^{2,H}(t_i,s)) (s-u)^{H-1/2} ds dB_k(u),
\end{align*}
where the last equality is justified by the stochastic Fubini theorem. Indeed,  
\begin{align*}
& \E\left[\int_{t_{i+1}}^T \int_{t_i}^{t_{i+1}} \left| \frac{\partial}{\partial x_k} P_{\frac{1}{2H}(s-t_i)^{2H}} f^a_{t_i}(s,W^{2,H}(t_i,s)) \right|^2 (s-u)^{2H-1} du ds \right]\\
&= \E\left[\int_{t_{i+1}}^T \int_{t_i}^{t_{i+1}} \left| \E_{t_i}\left[P_{\frac{1}{2H}(s-t_i)^{2H}} \frac{\partial}{\partial y_k} f(s,W^{2,H}(t_i,s))\right] \right|^2 (s-u)^{2H-1} du ds \right]\\
&\leq C \int_{t_{i+1}}^T \int_{t_i}^{t_{i+1}} (s-u)^{2H-1} du ds <+\infty. 
\end{align*}
Using this expression, the It\^o isometry and the independence of the disjoint increments of the Brownian motion, we get that 
\begin{align*}
&\hspace{-3em}\E\left[\left|\sum_{i=0}^{N-1} I_{1,3,k}(\ti,\tio)-\int_{t_i}^{t_{i+1}} \int_{t_{i}}^T \frac{\partial}{\partial y_k} P_{\frac{1}{2H}(s-t_i)^{2H}} f^a_{t_i}(s,W^{2,H}(t_i,s)) (s-t_i)^{H-1/2} ds dB_k(u)\right|^2\right]\\
&\hspace{-3em}\leq 2 \E\left[\left|\sum_{i=0}^{N-1} \int_{t_i}^{t_{i+1}} \int_{t_{i+1}}^T \frac{\partial}{\partial y_k} P_{\frac{1}{2H}(s-t_i)^{2H}} f^a_{t_i}(s,W^{2,H}(t_i,s)) \left((s-u)^{H-1/2} -(s-t_i)^{H-1/2}\right) ds dB_k(u) \right|^2\right]\\
&\hspace{-1em}+ 2 \E\left[\left|\sum_{i=0}^{N-1} \int_{t_i}^{t_{i+1}} \int_{t_{i}}^{t_{i+1}} \frac{\partial}{\partial y_k} P_{\frac{1}{2H}(s-t_i)^{2H}} f^a_{t_i}(s,W^{2,H}(t_i,s)) (s-t_i)^{H-1/2} ds dB_k(u) \right|^2\right]\\
&\hspace{-3em}= 2 \sum_{i=0}^{N-1} \int_{t_i}^{t_{i+1}} \E\left[\left|\int_{t_{i+1}}^T \frac{\partial}{\partial y_k} P_{\frac{1}{2H}(s-t_i)^{2H}} f^a_{t_i}(s,W^{2,H}(t_i,s)) \left((s-u)^{H-1/2} -(s-t_i)^{H-1/2}\right) ds \right|^2\right] du\\
&\hspace{-1em}+ 2 \sum_{i=0}^{N-1} \int_{t_i}^{t_{i+1}} \E\left[\left| \int_{t_{i}}^{t_{i+1}} \frac{\partial}{\partial y_k} P_{\frac{1}{2H}(s-t_i)^{2H}} f^a_{t_i}(s,W^{2,H}(t_i,s)) (s-t_i)^{H-1/2} ds \right|^2\right]du\\
&\hspace{-3em}\leq 2 C S_N,
\end{align*}
where 
$$ S_N:=\sum_{i=0}^{N-1} \int_{t_i}^{t_{i+1}} \left(\left|\int_{t_{i+1}}^T \left|(s-u)^{H-1/2} -(s-t_i)^{H-1/2}\right| ds \right|^2 +\left| \int_{t_{i}}^{t_{i+1}} (s-t_i)^{H-1/2} ds \right|^2\right)du.$$
A direct computation gives that $\lim_{N\to+\infty} S_N =0$. It remains to prove that the process
\begin{equation}\label{eq:procI13}
t \to \int_t^T P_{\frac{1}{2H}(s-t)^{2H}} \frac{\partial}{\partial x_k} f^a_t(s,W^{2,H}(s,t)) (s-t)^{H-1/2} ds,
\end{equation}
is continuous in $L^2(\Omega\times[0,T])$ in order to verifies the assumptions of \cite[Theorem 2.74]{jacod2006calcul} in order to deduce that
\begin{align*}
\sum_{i = 0}^{N-1} \int_{t_i}^{t_{i+1}} \int_{t_{i}}^T \frac{\partial}{\partial y_k} P_{\frac{1}{2H}(s-t_i)^{2H}} f^a_{t_i}(s,W^{2,H}(t_i,s)) (s-t_i)^{H-1/2} ds dB_k(u)
\\ \underset{N\to\infty}{\to}  \int_0^T \int_t^T \frac{\partial}{\partial y_k} P_{\frac{1}{2H}(s-t)^{2H}} f^a_{t}(s,W^{2,H}(t,s))(s-t)^{H-1/2}ds dB_k(t).
\end{align*}
First, we prove the domination assumption. Using the change of variable $u = s-t$ and the fact that $f$ is a smooth random field, we obtain the following estimate
\begin{align*}
\left|\int_t^T P_{\frac{1}{2H}(s-t)^{2H}} \frac{\partial}{\partial x_k} f^a_t(s,W^{2,H}(s,t)) (s-t)^{H-1/2} ds\right|
\\ = \left|\int_0^{T-t} P_{\frac{1}{2H}u^{2H}} \frac{\partial}{\partial x_k} f^a_{t}(u+t,W^{2,H}(u+t,t)) u^{H-1/2} du\right|
\\ \lesssim \int_0^{T-t} u^{H-1/2} du \leq (T-t)^{H+1/2}.
\end{align*}
We now turn to the continuity of the process \eqref{eq:procI13} itself. By the change of variable $u = s-t$, we essentially have to prove that $f^a_{t}(u+t,W^{2,H}(u+t,t))$ is continuous with respect to $t$. The only difficulty is the continuity of $t\to f^a_{t}(u,y)$ for any $(u,y)\in[0,T]\times\mathbb{R}^d$. Clark-Ocone's formula gives
\begin{equation*}
f(u,y) = \mathbb{E}[f(u,y)] + \sum_{j =1}^d \int_0^T \mathbb{E}_r[D_j( f(u,y)) (r)] dB_j(r),
\end{equation*}
then, we derive
\begin{equation*}
f^a_t(u,y) = \mathbb{E}[f(u,y)] + \sum_{j =1}^d \int_0^t \mathbb{E}_r[D_j( f(u,y)) (r)] dB_j(r),
\end{equation*}
which is continuous with respect to $t$ uniformly in $(u,y)$. This ends the proof of (\ref{eq:I13temp1}).
\\\\ 
\noindent
\textbf{Proof of (iii) : \\\\} 
For fixed $i \in\{0,\cdots,N-1\}$, $j,k$ in $\{1,\cdots,d\}$, $s\in[t_{i+1},T]$, we set 
$$\alpha_{i,j,k,s}(u):=\frac{\partial}{\partial y_k} P_{\frac{1}{2H}(s-t_i)^{2H}} g_j(s,u,W^{2,H}(t_i,s)),$$ 
$$ M_{i,j,k,s}(r):= \int_{t_i}^r \alpha_{i,j,k,s}(u) dB_j(u), \quad N_{i,k,s}(r):=\int_{t_i}^{r} (s-u)^{H-\frac12} dB_k(u), \; r\in[t_i,t_{i+1}].$$
so that $M_{i,j,k,s}$ and $N_{i,k,s}$ are continuous martingales. Note once again that since $f$ belongs to $\mathcal{S}_{ad}$, $\alpha_{i,j,k,s}(u)$ is uniformly (in $i,j,k,s,u$) bounded $\P$-a.s. Thus
$$ I_{1,4}(t_i,t_{i+1}) = \sum_{k=1}^d \sum_{j=1}^d \int_{t_{i+1}}^T M_{i,j,k,s}(t_{i+1}) N_{i,k,s}(t_{i+1}) ds.$$
The integration by parts formula for semimartingales implies that 
\begin{align}
\label{eq:temp2termI14}
&M_{i,j,k,s}(t_{i+1}) N_{i,k,s}(t_{i+1}) - \textbf{1}_{j=k} \int_{t_i}^{t_{i+1}} \alpha_{i,j,k,s}(u) (s-u)^{H-\frac12} du\nonumber\\
&=\int_{t_i}^{t_{i+1}} M_{i,j,k,s}(r) dN_{i,k,s}(r) + \int_{t_i}^{t_{i+1}} N_{i,k,s}(r) dM_{i,j,k,s}(r).
\end{align}
We show below that both terms in the right hand side do not contribute to the limit. Indeed, using the fact that the co-variation $[B_j(\cdot),B_{j'}(\cdot)]=0$ for any $j\neq j'$, we get 
\begin{align}
\label{eq:temp1termI14}
&\E\left[\left|\sum_{i=0}^{N-1} \int_{t_{i+1}}^T \sum_{k=1}^d \sum_{j=1}^d \int_{t_i}^{t_{i+1}} M_{i,j,k,s}(r) dN_{i,k,s}(r) ds\right|^2\right] \nonumber\\
&=\sum_{i,i'=0}^{N-1} \sum_{k,k'=1}^d \sum_{j,j'=1}^d \int_{t_{i+1}}^T \int_{t_{i'+1}}^T \E\left[\int_{t_i}^{t_{i+1}} M_{i,j,k,s}(r) dN_{i,k,s}(r) \int_{t_i'}^{t_{i'+1}} M_{i',j',k',s'}(r') dN_{i',k',s'}(r')\right] ds ds'\nonumber\\
&=\sum_{i=0}^{N-1} \sum_{k=1}^d \sum_{j,j'=1}^d \int_{t_{i+1}}^T \int_{t_{i+1}}^T \int_{t_i}^{t_{i+1}} \E\left[ M_{i,j,k,s}(r) M_{i,j',k,s'}(r) \right] (s-r)^{H-1/2} (s'-r)^{H-1/2} dr ds ds'\nonumber\\
&=\sum_{i=0}^{N-1} \sum_{k=1}^d \sum_{j=1}^d \int_{t_{i+1}}^T \int_{t_{i+1}}^T \int_{t_i}^{t_{i+1}} \int_{t_i}^r \E[\alpha_{i,j,k,s}(u) \alpha_{i,j,k,s'}(u)] du (s-r)^{H-1/2} (s'-r)^{H-1/2} dr ds ds'\nonumber\\
&\leq \frac{C}{N} \sum_{i=0}^{N-1} \int_{t_{i+1}}^T \int_{t_{i+1}}^T \int_{t_i}^{t_{i+1}} (s-r)^{H-1/2} (s'-r)^{H-1/2} dr ds ds' \nonumber\\
&= \frac{C}{N} \sum_{i=0}^{N-1}\int_{t_i}^{t_{i+1}} \left(\int_{t_{i+1}}^T (s-r)^{H-1/2} ds\right)^2 dr\nonumber\\
&\underset{N\to+\infty}{\longrightarrow} 0.
\end{align}
Now we turn to the analysis of the the second term in the right hand side of (\ref{eq:temp2termI14}). The first arguments follow the same line as for the term above (using mainly the independence of the components of the Brownian motion $B$). Indeed, we have : 
\begin{align}
\label{eq:temp3termI14}
&\E\left[\left|\sum_{i=0}^{N-1} \int_{t_{i+1}}^T \sum_{k=1}^d \sum_{j=1}^d \int_{t_i}^{t_{i+1}} N_{i,k,s}(r) dM_{i,j,k,s}(r) ds\right|^2\right] \nonumber\\
&=\sum_{i=0}^{N-1} \sum_{k,k'=1}^d \sum_{j=1}^d \int_{t_{i+1}}^T \int_{t_{i+1}}^T \int_{t_i}^{t_{i+1}} \E\left[N_{i,k,s}(r) N_{i,k',s'}(r) \alpha_{i,j,k,s}(r) \alpha_{i,j,k',s'}(r)\right] dr ds ds'\nonumber\\
&\leq C \sum_{i=0}^{N-1} \sum_{k,k'=1}^d \sum_{j=1}^d \int_{t_{i+1}}^T \int_{t_{i+1}}^T \int_{t_i}^{t_{i+1}} \E\left[|N_{i,k,s}(r) N_{i,k',s'}(r)|\right] dr ds ds'\nonumber\\
&= C \sum_{i=0}^{N-1} \sum_{k,k'=1}^d \int_{t_{i+1}}^T \int_{t_{i+1}}^T \int_{t_i}^{t_{i+1}} \left(\int_{\ti}^r (s-v)^{2H-1} dv \int_{\ti}^r (s'-v)^{2H-1} dv\right)^{1/2} dr ds ds'.
\end{align} 
So plugging this estimate in (\ref{eq:temp3termI14}), we get 
\begin{align}
\label{eq:temp3bistermI14}
&\E\left[\left|\sum_{i=0}^{N-1} \int_{t_{i+1}}^T \sum_{k=1}^d \sum_{j=1}^d \int_{t_i}^{t_{i+1}} N_{i,k,s}(r) dM_{i,j,k,s}(r) ds\right|^2\right] \nonumber\\
&\leq C \sum_{i=0}^{N-1} \int_{t_i}^{t_{i+1}} \left(\int_{t_{i+1}}^T \left(\int_{\ti}^r (s-v)^{2H-1} dv\right)^{1/2} ds\right)^2 dr\nonumber\\
&\leq C \sum_{i=0}^{N-1} \int_{t_i}^{t_{i+1}} \int_{t_{i+1}}^T \int_{\ti}^r (s-v)^{2H-1} dv ds dr\nonumber\\
&\underset{N\to+\infty}{\longrightarrow} 0.
\end{align}
So to summarize, Relations (\ref{eq:temp2termI14}), (\ref{eq:temp1termI14}) and (\ref{eq:temp3bistermI14}) imply that : 
$$ \lim_{N\to +\infty} \E\left[\left|\sum_{i=0}^{N-1} I_{1,4}(t_i,t_{i+1}) - \sum_{i=0}^{N-1} \sum_{j=1}^d \int_{t_{i+1}}^T \int_{t_i}^{t_{i+1}} \alpha_{i,j,j,s}(u) (s-u)^{H-\frac12} du ds\right|^2\right]=0.$$
However we have that : 
\begin{align*}
&\hspace{-2em} \E\left[\left|\sum_{j=1}^d \left(\sum_{i=0}^{N-1} \int_{t_{i+1}}^T \int_{t_i}^{t_{i+1}} \alpha_{i,j,j,s}(u) (s-u)^{H-\frac12} du ds - \int_0^T  \int_{t}^T \partial_{y_j} P_{\frac{1}{2H}(s-t)^{2H}} g_j(s,t,W^{2,H}(t,s)) (s-t)^{H-1/2}dsdt\right) \right|^2\right]\\
&\hspace{-2em}=\E\left[\left|\sum_{j=1}^d \sum_{i=0}^{N-1} \int_{t_i}^{t_{i+1}} \left( \int_{t_{i+1}}^T \alpha_{i,j,j,s}(u) (s-u)^{H-\frac12} ds - \int_{u}^T \partial_{y_j} P_{\frac{1}{2H}(s-u)^{2H}} g_j(s,u,W^{2,H}(u,s)) (s-u)^{H-1/2}ds\right)du \right|^2\right]\\
&\hspace{-2em}=\E\left[\left|\sum_{j=1}^d \sum_{i=0}^{N-1} \int_{t_i}^{t_{i+1}} \left( \int_{u}^T \beta_{i,j,s}(u) (s-u)^{H-\frac12} ds - \int_{u}^{t_{i+1}} \alpha_{i,j,j,s}(u) (s-u)^{H-\frac12} ds\right)du \right|^2\right]\\
&\hspace{-2em}\leq 2 \E\left[\left|\sum_{j=1}^d \sum_{i=0}^{N-1} \int_{t_i}^{t_{i+1}} \int_{u}^T \beta_{i,j,s}(u) (s-u)^{H-\frac12} ds du \right|^2\right]\\
&\hspace{-2em}+ 2 \E\left[\left|\sum_{j=1}^d \sum_{i=0}^{N-1} \int_{t_i}^{t_{i+1}} \int_{u}^{t_{i+1}} \alpha_{i,j,j,s}(u) (s-u)^{H-\frac12} ds du \right|^2\right]\\
&\hspace{-2em}=:2 \left(I_{1,4,1} + I_{1,4,2}\right). 
\end{align*} 
with 
\begin{align*}
\beta_{i,j,s}(u)&:=\alpha_{i,j,j,s}(u)-\partial_{y_j} P_{\frac{1}{2H}(s-u)^{2H}} g_j(s,u,W^{2,H}(u,s)).
\end{align*}
The proof of (iii) is then established if we prove that 
\begin{equation}
\label{eq:temp5I14}
\lim_{N\to+\infty} I_{1,4,1} + I_{1,4,2}=0.
\end{equation}
Note first that : 
\begin{align}
\label{eq:betatemp1}
&\beta_{i,j,s}(u)\nonumber\\
&= \partial_{x_j} \left(P_{\frac{1}{2H}(s-t_i)^{2H}} g_j(s,u,W^{2,H}(t_i,s))-P_{\frac{1}{2H}(s-u)^{2H}} g_j(s,u,W^{2,H}(u,s))\right)\nonumber\\
&= \left(P_{\frac{1}{2H}(s-t_i)^{2H}}-P_{\frac{1}{2H}(s-u)^{2H}}\right) \partial_{x_j} g_j(s,u,W^{2,H}(t_i,s))\nonumber\\
&+ P_{\frac{1}{2H}(s-u)^{2H}} \left(\partial_{x_j} g_j(s,u,W^{2,H}(t_i,s))-\partial_{x_j} g_j(s,u,W^{2,H}(u,s))\right)\nonumber\\
&=-\frac12 \int_{t_i}^u \Delta P_{\frac{1}{2H}(s-r)^{2H}} \partial_{x_j} g_j(s,u,W^{2,H}(t_i,s)) dr \nonumber\\
&+ \int_0^1 P_{\frac{1}{2H}(s-u)^{2H}} \nabla \partial_{x_j} g_j(s,u,W^{2,H}(t_i,s,\theta)) d\theta \cdot (W^{2,H}(t_i,s)-W^{2,H}(u,s)),
\end{align}
where we recall Notation (\ref{eq:inctheta}). Using once again the fact that $f$ belongs to $\mathcal{S}_{ad}$, we immediately obtain that 
\begin{equation}
\label{eq:boundbeta}
|\beta_{i,j,s}(u)| \leq C \left((u-t_i) + \sum_{k=1}^d \left|W^{2,H}(\ti,s)-W^{2,H}(u,s)\right|\right),
\end{equation}
from which we deduce that (using (\ref{eq:incfrac}))
\begin{align*}
&\hspace{-2em}\left(I_{1,4,1}\right)^{1/2}\\
&\hspace{-2em}\leq C \sum_{j=1}^d \sum_{i=0}^{N-1} \int_{t_i}^{t_{i+1}} \int_{u}^T \left((u-t_i)+\sum_{k=1}^d \E\left[\left|W^{2,H}(\ti,s)-W^{2,H}(u,s)\right|^2\right]^{1/2}\right) (s-u)^{H-\frac12} ds du \\
&\hspace{-2em}\leq C \sum_{i=0}^{N-1} \int_{t_i}^{t_{i+1}} \int_{u}^T \left((u-t_i)+|u-t_i|^{\min\{H,1/2\}}\right) (s-u)^{H-\frac12} ds du \\
&\underset{N\to+\infty}{\longrightarrow} 0. 
\end{align*}
Thus 
$$ \lim_{N\to+\infty} I_{1,4,1} = 0.$$
The convergence of the term $I_{1,4,2}$ is easy to handle as :
\begin{align*}
\left(I_{1,4,2}\right)^{1/2} & \leq \sum_{j=1}^d \sum_{i=0}^{N-1} \int_{t_i}^{t_{i+1}} \int_{u}^{t_{i+1}} (s-u)^{H-\frac12} \E\left[\left| \alpha_{i,j,j,s}(u)\right|^2\right]^{1/2} ds du\\
&\leq C \sum_{i=0}^{N-1} \int_{t_i}^{t_{i+1}} \int_{u}^{t_{i+1}} (s-u)^{H-\frac12} ds du \\
&\underset{N\to+\infty}{\longrightarrow} 0.
\end{align*}  
So (\ref{eq:temp5I14}) is proved.\\\\ 
\noindent

\textbf{Proof of (iv) : \\\\}
Recall that 
$$ I_{2,5}(t_i,t_{i+1})= -\sumj \int_{t_i}^{t_{i+1}} \int_u^T P_{\frac{1}{2H}(s-t_{i})^{2H}} g_j(s,u,W^{2,H}(t_{i},s)) ds dB_j(u).$$
Hence : 
\begin{align*}
&\sum_{i=0}^{N-1} I_{2,5}(t_i,t_{i+1}) + \int_{0}^{T} \int_u^T P_{\frac{1}{2H}(s-u)^{2H}} g_j(s,u,W^{2,H}(t_{i},s)) ds dB_j(u)\\
&=-\sumj \sum_{i=0}^{N-1} \int_{t_i}^{t_{i+1}} \int_u^T \left(P_{\frac{1}{2H}(s-t_{i})^{2H}}-P_{\frac{1}{2H}(s-u)^{2H}}\right) g_j(s,u,W^{2,H}(t_{i},s)) ds dB_j(u)\\
&-\sumj \sum_{i=0}^{N-1} \int_{t_i}^{t_{i+1}} \int_u^T P_{\frac{1}{2H}(s-u)^{2H}} (g_j(s,u,W^{2,H}(t_{i},s))-g_j(s,u,W^{2,H}(u,s))) ds dB_j(u)\\
&=\sumj \sum_{i=0}^{N-1} \int_{t_i}^{t_{i+1}} \int_u^T \gamma_{s,t_i}(u) ds dB_j(u)
\end{align*}
with 
\begin{align*}
&\gamma_{s,t_i}(u)\\
&:= \left(P_{\frac{1}{2H}(s-t_{i})^{2H}}-P_{\frac{1}{2H}(s-u)^{2H}}\right) g_j(s,u,W^{2,H}(t_{i},s)) + P_{\frac{1}{2H}(s-u)^{2H}} \left[g_j(s,u,W^{2,H}(t_{i},s))-g_j(s,u,W^{2,H}(u,s))\right].
\end{align*}
Hence using the It\^o isometry, 
\begin{align*}
&\E\left[\left|\sum_{i=0}^{N-1} I_{2,5}(t_i,t_{i+1}) + \int_{0}^{T} \int_u^T P_{\frac{1}{2H}(s-u)^{2H}} g_j(s,u,W^{2,H}(t_{i},s)) ds dB_j(u)\right|\right]^{1/2}\\
&\leq \sumj \sum_{i=0}^{N-1} \int_{t_i}^{t_{i+1}} \E\left[\left| \int_u^T\gamma_{s,t_i}(u) ds \right|^2\right]^{1/2} du\\
&\leq \sumj \sum_{i=0}^{N-1} \int_{t_i}^{t_{i+1}} \int_u^T \E\left[\left|\gamma_{s,t_i}(u) \right|^2\right]^{1/2} ds du.
\end{align*}
Up to the gradient, the quantity $\gamma_{s,t_i}$ is very similar to $\beta_{i,j,s}$ defined in (\ref{eq:betatemp1}) and using (\ref{eq:boundbeta}) and (\ref{eq:incfrac}), we get 
\begin{align*}
&\E\left[\left|\sum_{i=0}^{N-1} I_{2,5}(t_i,t_{i+1}) + \int_{0}^{T} \int_u^T P_{\frac{1}{2H}(s-u)^{2H}} g_j(s,u,W^{2,H}(t_{i},s)) ds dB_j(u)\right|\right]^{1/2}\\
&\leq C \sum_{i=0}^{N-1} \int_{t_i}^{t_{i+1}} \int_u^T \left( (u-t_i) + |t_{i+1}-t_i|^{\min\{H,1/2\}}\right) ds du\\
&\underset{N\to+\infty}{\longrightarrow} 0.
\end{align*}
\end{proof}

\begin{lemma}
\label{lemma:tec2}
We use notations introduced in the proof of Theorem \ref{thm:main}, the following convergences hold true in $L^2(\Omega)$ : 
\begin{itemize}
\item[(i)] $$\lim_{N\to+\infty} \sum_{i = 0}^{N-1} I_{1,2}(t_i,t_{i+1}) + \lim_{N\to+\infty} \sum_{i = 0}^{N-1} I_{1,6}(t_i,t_{i+1})=0,$$
\item[(ii)] $$\lim_{N\rightarrow \infty} \sum_{i = 0}^{N-1} I_{2,1}(t_i,t_{i+1}) + \lim_{N\to+\infty} \sum_{i = 0}^{N-1} I_{2,3}(t_i,t_{i+1})=0,$$
\item[(iii)] $$\lim_{N\rightarrow \infty} \sum_{i = 0}^{N-1} I_{1,7}(\ti,\tio) = 0,$$
\item[(iv)] $$\lim_{N\rightarrow \infty} \sum_{i = 0}^{N-1} I_{1,8}(\ti,\tio) = 0,$$
\item[(v)] $$\lim_{N\rightarrow \infty} \sum_{i = 0}^{N-1} I_{1,9}(\ti,\tio) = 0,$$
\item[(vi)] $$\lim_{N\rightarrow \infty} \sum_{i = 0}^{N-1} I_{1,10}(\ti,\tio) = 0.$$
\end{itemize}
\end{lemma}

\begin{proof}
\textbf{Proof of (i)\\\\}
\noindent
As we will see some cancellations appear among the terms in the rest. We start with one of these cancellations, that is we first prove that 
\begin{equation}
\label{eq:cancel1}
\lim_{N\to+\infty} \sum_{i = 0}^N I_{1,2}(t_i,t_{i+1}) + \lim_{N\to+\infty} \sum_{i = 0}^N I_{1,6}(t_i,t_{i+1}) \overset{{L^2(\Omega)}}{=} 0. 
\end{equation}
Recall first that
\begin{align}
\label{eq:I12temp1}
I_{1,2}(t_i,t_{i+1}) &= \int_{t_{i+1}}^T \left[P_{\frac{1}{2H}(s-t_{i+1})^{2H}} - P_{\frac{1}{2H}(s-t_i)^{2H}}\right]f(s,W^{2,H}(t_{i+1},s))ds\nonumber\\ 
&= -\frac 1 2 \int_{t_{i+1}}^T\int_{t_i}^{t_{i+1}} (s-u)^{2H-1}\Delta P_{\frac{1}{2H}(s-u)^{2H}}f(s,W^{2,H}(t_{i+1},s))du ds\nonumber\\
&= -\frac 1 2 \int_{t_i}^{t_{i+1}} \int_{t_{i+1}}^T (s-u)^{2H-1}\Delta P_{\frac{1}{2H}(s-u)^{2H}}f(s,W^{2,H}(t_{i+1},s))ds du.
\end{align}
Concerning the term $I_{1,6}(t_i,t_{i+1})$ we have
\begin{align*}
&I_{1,6}(t_i,t_{i+1})\\
&= \frac12 \sum_{k=1}^d \int_{t_{i+1}}^T \frac{\partial^2}{\partial y_k^2} P_{\frac{1}{2H}(s-t_i)^{2H}} f\left(s,W^{2,H}(t_i,s)\right) \; \left(\delta_{k,i,s}(W^{2,H})\right)^2 ds.
\end{align*}
So 
\begin{align*}
&I_{1,6}(t_i,t_{i+1})\\
&= \frac12 \sum_{k=1}^d \int_{t_{i+1}}^T \frac{\partial^2}{\partial y_k^2} P_{\frac{1}{2H}(s-t_i)^{2H}} f\left(s,W^{2,H}(t_i,s)\right) \; \left(\left|\delta_{k,i,s}(W^{2,H})\right|^2-\int_{t_i}^{t_{i+1}} (s-u)^{2H-1} du\right) ds\\
&+\frac12 \sum_{k=1}^d \int_{t_i}^{t_{i+1}} \int_{t_{i+1}}^T (s-u)^{2H-1} \Delta P_{\frac{1}{2H}(s-t_i)^{2H}} f\left(s,W^{2,H}(t_i,s)\right) ds du\\
&=:I_{1,6,1}(t_i,t_{i+1})+I_{1,6,2}(t_i,t_{i+1}).
\end{align*}
As a consequence using (\ref{eq:I12temp1}), and letting : 
\begin{align}
\label{eq:I12temp2}
&\hspace{-2em} A(\ti,\tio)\nonumber\\
&\hspace{-2em} := \frac12 \int_{t_i}^{t_{i+1}} \int_{t_{i+1}}^T (s-u)^{2H-1} \left( \Delta P_{\frac{1}{2H}(s-t_i)^{2H}} f\left(s,W^{2,H}(t_i,s)\right)-\Delta P_{\frac{1}{2H}(s-u)^{2H}}f(s,W^{2,H}(t_{i+1},s))\right) ds du,
\end{align}
$I_{1,2}(\ti,\tio)+I_{1,6}(\ti,\tio)$ writes down as 
\begin{align*}
&I_{1,2}(\ti,\tio)+I_{1,6}(\ti,\tio)=I_{1,6,1}(t_i,t_{i+1})+ A(\ti,\tio).
\end{align*}
Hence, (\ref{eq:cancel1}) is proved if we prove
\begin{equation}
\label{eq:I12temp3}
\lim_{N\to +\infty}\E\left[\left| \sum_{i=0}^{N-1} I_{1,6,1}(t_i,t_{i+1}) \right|^2\right] = 0,
\end{equation}
and
\begin{equation}
\label{eq:I12temp4}
\lim_{N\to +\infty}\E\left[\left| \sum_{i=0}^{N-1} A(t_i,t_{i+1}) \right|^2\right] =0.
\end{equation}
We start with an analysis of Term $I_{1,6,1}(\ti,\tio)$, and we write $I_{1,6,1}(\ti,\tio)=\frac12 \sum_{k=1}^d I_{1,6,1,k}(\ti,\tio)$, with 
$$ I_{1,6,1,k}(\ti,\tio):= \int_{t_{i+1}}^T \frac{\partial^2}{\partial y_k^2} P_{\frac{1}{2H}(s-t_i)^{2H}} f\left(s,W^{2,H}(t_i,s)\right) \; \left(\left|\delta_{k,i,s}(W^{2,H})\right|^2-\int_{t_i}^{t_{i+1}} (s-u)^{2H-1} du\right) ds.$$
We have by letting $\rho_{i,s}:=\frac{\partial^2}{\partial y_k^2} P_{\frac{1}{2H}(s-t_i)^{2H}} f\left(s,W^{2,H}(t_i,s)\right)$, and 
\begin{equation}
\label{eq:epsilon}
\epsilon_{i,s,k}:=\left|\delta_{k,i,s}(W^{2,H})\right|^2-\int_{t_i}^{t_{i+1}} (s-u)^{2H-1} du.
\end{equation}
We have
\begin{align*}
&\E\left[\left|\sum_{i=0}^{N-1} I_{1,6,1,k}(t_i,t_{i+1})\right|^2\right]\\
&=2\sum_{i,i'=0;i<i'}^{N-1} \int_{t_{i+1}}^T \int_{t_{i'+1}}^T \E\left[ \rho_{i,s} \rho_{i',s'} \epsilon_{i,s,k} \underbrace{\E_{t_{i'}}\left[\epsilon_{i',s',k}\right]}_{=0} \right] ds ds'\\
&+\sum_{i=0}^{N-1} \int_{t_{i+1}}^T \int_{t_{i+1}}^T \E\left[ \rho_{i,s} \rho_{i,s'} \epsilon_{i,s,k} \epsilon_{i,s',k} \right] ds ds'\\
&\leq C\sum_{i=0}^{N-1} \int_{t_{i+1}}^T \int_{t_{i+1}}^T \E\left[|\epsilon_{i,s,k} \epsilon_{i,s',k}| \right] ds ds'\\
&\leq C\sum_{i=0}^{N-1} \left(\int_{t_{i+1}}^T \int_{t_i}^{\tio} (s-v)^{2H-1} dv ds\right)^2\\
&\underset{N\to+\infty}{\longrightarrow} 0.
\end{align*}
So (\ref{eq:I12temp3}) is proved. 
Convergence (\ref{eq:I12temp4}) is obtained as follows. Note first that : 
\begin{align*}
&\left|\Delta P_{\frac{1}{2H}(s-t_i)^{2H}} f\left(s,W^{2,H}(t_i,s)\right) -\Delta P_{\frac{1}{2H}(s-u)^{2H}}f(s,W^{2,H}(t_{i+1},s))\right|\\
&\leq \left|\Delta P_{\frac{1}{2H}(s-t_i)^{2H}} f\left(s,W^{2,H}(t_i,s)\right) -\Delta P_{\frac{1}{2H}(s-t_i)^{2H}}f(s,W^{2,H}(t_{i+1},s))\right|\\
&+\left|\Delta P_{\frac{1}{2H}(s-t_i)^{2H}}f(s,W^{2,H}(t_{i+1},s))-\Delta P_{\frac{1}{2H}(s-u)^{2H}}f(s,W^{2,H}(t_{i+1},s))\right|\\
&\leq C \sum_{k=1}^d \left|\delta_{k,i,s}(W^{2,H})\right|+C \int_{t_i}^u (s-r)^{2H-1} dr, 
\end{align*} 
where $C$ depends on the sup norms of partial derivatives of $\varphi$ (recall (\ref{eq:defif})) up to order $4$ and where we have used the definition of the Heat semigroup as in (\ref{eq:I12temp1}). Thus, since 
$$ \E\left[\delta_{k,i,s}(W^{2,H}) \delta_{k,i',s}(W^{2,H})\right]=0, \forall i\neq i', $$
we have (recalling (\ref{eq:incfrac}))
\begin{align*}
&\E\left[\left| \sum_{i=0}^{N-1} A(t_i,t_{i+1}) \right|^2\right]\\
&\leq C \sum_{k=1}^d \E\left[\left| \sum_{i=0}^{N-1} \int_{t_i}^{t_{i+1}} \int_{t_{i+1}}^T (s-u)^{2H-1} \left|\delta_{k,i,s}(W^{2,H})\right| ds du \right|^2\right]\\
&+ C \left|\sum_{i=0}^{N-1} \int_{t_i}^{t_{i+1}} \int_{t_{i+1}}^T (s-u)^{2H-1} \int_{t_i}^u (s-r)^{2H-1} dr ds du\right|^2 \\
&\leq C \sum_{k=1}^d \left(\sum_{i=0}^{N-1} \int_{t_i}^{t_{i+1}} \int_{t_{i+1}}^T (s-u)^{2H-1} |t_{i+1}-t_i|^{\min\{H,1/2\}} ds du \right)^2\\
&\underset{N\to+\infty}{\longrightarrow} 0,
\end{align*}
which proves (\ref{eq:I12temp4}).\\\\
\noindent
\textbf{Proof of (ii)\\\\}
The second cancellation is the following 
\begin{equation}
\label{eq:cancel2}
\lim_{N\rightarrow \infty} \sum_{i = 0}^N I_{2,1}(t_i,t_{i+1})+I_{2,3}(t_i,t_{i+1}) = 0.
\end{equation}
Before getting into the computations, it is worth noting that $I_{2,1}(t_i,t_{i+1})$ (respectively $I_{2,3}(t_i,t_{i+1})$) has the same structure (up to the Brownian integral) than $I_{1,2}(t_i,t_{i+1})$ (respectively $I_{1,6}(t_i,t_{i+1})$ and $I_{1,7}(t_i,t_{i+1})$). So the proof will follow the same lines as in the one of (i). For the sake of completeness, we tough provide the main arguments. 
Recall that
\begin{align*}
&\hspace{-2em}I_{2,1}(t_i,t_{i+1})\\
&\hspace{-2em}=\sum_{j=1}^d \int_{t_{i+1}}^T \int_u^T \left[ P_{\frac{1}{2H}(s-t_{i+1})^{2H}} - P_{\frac{1}{2H}(s-t_i)^{2H}} \right] g_j(s,u,W^{2,H}(t_{i+1},s))ds dB_j(u) \\
&\hspace{-2em}=-\frac12\sum_{j=1}^d \int_{t_{i+1}}^T \int_u^T \int_{t_i}^{t_{i+1}} (s-r)^{2H-1} \Delta P_{\frac{1}{2H}(s-r)^{2H}} g_j(s,u,W^{2,H}(t_{i},s)) dr ds dB_j(u)\\
&\hspace{-2em}-\frac12\sum_{j=1}^d \int_{t_{i+1}}^T \int_u^T \int_{t_i}^{t_{i+1}} (s-r)^{2H-1} \Delta P_{\frac{1}{2H}(s-r)^{2H}}\left(g_j(s,u,W^{2,H}(t_{i+1},s))- g_j(s,u,W^{2,H}(t_{i},s))\right) dr ds dB_j(u)\\
&\hspace{-2em}=-\frac12\sum_{j=1}^d \int_{t_{i+1}}^T \int_u^T \int_{t_i}^{t_{i+1}} (s-r)^{2H-1} \Delta P_{\frac{1}{2H}(s-t_i)^{2H}} g_j(s,u,W^{2,H}(t_{i},s)) dr ds dB_j(u)\\
&\hspace{-2em}-\frac12\sum_{j=1}^d \int_{t_{i+1}}^T \int_u^T \int_{t_i}^{t_{i+1}} (s-r)^{2H-1} \left(\Delta P_{\frac{1}{2H}(s-r)^{2H}} -\Delta P_{\frac{1}{2H}(s-t_i)^{2H}} \right)g_j(s,u,W^{2,H}(t_{i},s)) dr ds dB_j(u)\\
&\hspace{-2em}-\frac12\sum_{j=1}^d \sum_{k,\ell=1}^d \int_{t_{i+1}}^T \int_u^T \int_{t_i}^{t_{i+1}} (s-r)^{2H-1} \int_0^1 \frac{\partial^3}{\partial y_k^2 \partial_{y_\ell}} P_{\frac{1}{2H}(s-r)^{2H}} g_j(s,u,W^{2,H}(t_{i},s,\theta)) d\theta \; \delta_{\ell,i,s}(W^{2,H}) dr ds dB_j(u),
\end{align*}
where we recall Notation (\ref{eq:inctheta}).
In addition   
\begin{align*}
&I_{2,3}(t_i,t_{i+1})\\
&=\frac12 \sum_{j=1}^d \int_{t_{i+1}}^T \int_u^T P_{\frac{1}{2H}(s-t_i)^{2H}} \nabla^2 g_j\left(s,u,W^{2,H}(t_i,s) \right) \cdot \left(\delta_{i,s}(W^{2,H})\right)^2 ds dB_j(u) \\
&=\frac12 \sum_{j=1}^d \sum_{k,\ell=1; k\neq \ell}^d \int_{t_{i+1}}^T \int_u^T \frac{\partial^2}{\partial x_k \partial x_\ell} P_{\frac{1}{2H}(s-t_i)^{2H}} g_j\left(s,u,W^{2,H}(t_i,s) \right) \delta_{i,k,s}(W^{2,H}) \delta_{i,\ell,s}(W^{2,H}) ds dB_j(u) \\
&+\frac12 \sum_{j=1}^d \sum_{k=1}^d \int_{t_{i+1}}^T \int_u^T \frac{\partial^2}{\partial y_k^2} P_{\frac{1}{2H}(s-t_i)^{2H}} g_j\left(s,u,W^{2,H}(t_i,s) \right) \left(\delta_{i,k,s}(W^{2,H})\right)^2 ds dB_j(u).
\end{align*}
Hence 
\begin{align*}
&\hspace{-2em}I_{2,1}(t_i,t_{i+1})+I_{2,3}(t_i,t_{i+1})\\
&\hspace{-2em}=\frac12 \sum_{j=1}^d \sum_{k,\ell=1; k\neq \ell}^d \int_{t_{i+1}}^T \int_u^T \frac{\partial^2}{\partial x_k \partial x_\ell} P_{\frac{1}{2H}(s-t_i)^{2H}} g_j\left(s,u,W^{2,H}(t_i,s)\right) \delta_{i,k,s}(W^{2,H}) \delta_{i,\ell,s}(W^{2,H}) ds dB_j(u) \\
&\hspace{-2em}+\frac12 \sum_{j=1}^d \sum_{k=1}^d \int_{t_{i+1}}^T \int_u^T \frac{\partial^2}{\partial x_k^2} P_{\frac{1}{2H}(s-t_i)^{2H}} g_j(s,u,W^{2,H}(t_i,s)) \left[\left(\delta_{i,k,s}(W^{2,H})\right)^2-\int_{t_i}^{t_{i+1}} (s-r)^{2H-1} dr\right] ds dB_j(u)\\
&\hspace{-2em}-\frac12\sum_{j=1}^d \sum_{k,\ell=1}^d \int_{t_{i+1}}^T \int_u^T \int_{t_i}^{t_{i+1}} (s-r)^{2H-1} \int_0^1 \frac{\partial^3}{\partial x_k^2 \partial_{x_\ell}} P_{\frac{1}{2H}(s-r)^{2H}} g_j(s,u,W^{2,H}(t_{i},s,\theta)) d\theta \; \delta_{\ell,i,s}(W^{2,H}) dr ds dB_j(u)\\
&\hspace{-2em}-\frac12\sum_{j=1}^d \int_{t_{i+1}}^T \int_u^T \int_{t_i}^{t_{i+1}} (s-r)^{2H-1} \left(\Delta P_{\frac{1}{2H}(s-r)^{2H}} -\Delta P_{\frac{1}{2H}(s-t_i)^{2H}} \right)g_j(s,u,W^{2,H}(t_{i},s)) dr ds dB_j(u)\\
&\hspace{-2em}=:C_1(\ti,\tio)+C_2(\ti,\tio)+C_3(\ti,\tio)+C_4(\ti,\tio).
\end{align*}
So obviously, (\ref{eq:cancel2}) is proved if we prove that 
\begin{equation}
\label{eq:eq:cancel2temp1}
\lim_{N\to+\infty} \E\left[\left|\sum_{i=0}^{N-1} C_r(\ti,\tio)\right|^2\right] =0, \quad \forall r\in \{1,2,3,4\}.
\end{equation}
These three terms are of similar form and their treatment will follow the similar scheme, so we give all the details for $C_1(\ti,\tio)$ and present only the key ingredients for $C_2(\ti,\tio)$ and $C_3(\ti,\tio)$. Hence we start with $C_1(\ti,\tio)$.\\ 
Set $\mu_{s,i,k,\ell,u}:=\frac{\partial^2}{\partial x_k \partial x_\ell} P_{\frac{1}{2H}(s-t_i)^{2H}} g_j\left(s,u,W^{2,H}(t_i,s) \right)$.
We write $C_1(\ti,\tio)$ as 
$$C_1(\ti,\tio) = \sum_{j=1}^d \sum_{k,\ell=1;k\neq \ell}^d C_{1,j,k,\ell}(\ti,\tio)$$ 
with obvious notations. We have for $j,k,\ell$ (with $k\neq \ell)$,
\begin{align*}
&\E\left[\left|\sum_{i=0}^{N-1} C_{1,j,k,\ell}(\ti,\tio)\right|^2\right]\\
&\hspace{-4em}= 2\sum_{i,i'=0; i< i'}^{N-1} \int_{t_{i'+1}}^T \int_u^T \int_u^T \E\left[\mu_{s,i,k,\ell,u} \mu_{s',i',k,\ell,u}\delta_{i,k,s}(W^{2,H}) \delta_{i,\ell,s}(W^{2,H}) \underbrace{\E_{t_{i'}}\left[\delta_{i',k,s'}(W^{2,H}) \delta_{i',\ell,s'}(W^{2,H})\right]}_{=0} \right] ds ds' du\\
&\hspace{-4em}+ \sum_{i=0}^{N-1} \int_{t_{i+1}}^T \int_u^T \int_u^T \E\left[\mu_{s,i,u} \mu_{s',i,u}\delta_{i,k,s}(W^{2,H}) \delta_{i,\ell,s}(W^{2,H}) \delta_{i,k,s'}(W^{2,H}) \delta_{i,\ell,s'}(W^{2,H}) \right] ds ds' du\\
&\hspace{-4em}\leq C \sum_{i=0}^{N-1} \int_{t_{i+1}}^T \left(\int_u^T \int_{t_i}^{\tio} (s-v)^{2H-1} dv ds\right)^2 du\\
&\hspace{-4em}\leq C \sum_{i=0}^{N-1} \left(\int_{\tio}^T \int_{t_i}^{\tio} (s-v)^{2H-1} dv ds\right)^2\\
&\hspace{-4em}\underset{N\to+\infty}{\longrightarrow} 0.
\end{align*}
With the previous notation and using Notation (\ref{eq:epsilon}), 
$$ C_2(\ti,\tio)= \frac12 \sum_{j=1}^d \sum_{k=1}^d \int_{t_{i+1}}^T \int_u^T \mu_{s,i,k,k,u} \epsilon_{i,s,k} ds dB_j(u).$$
So we have 
\begin{align*}
&\E\left[\left|\sum_{i=0}^{N-1} C_2(\ti,\tio)\right|^2\right]\\
&\hspace{-4em}\leq C \sum_{j=1}^d \sum_{k=1}^d \sum_{i,i'=0}^{N-1} \int_{t_{i'+1}\vee t_{i+1}}^T \int_u^T \int_u^T \E\left[\epsilon_{i,s,k} \epsilon_{i',s,k'} \right] ds ds' du\\
&\hspace{-4em}\leq C \left(\sum_{i=0}^{N-1} \left(\int_{\tio}^T \int_{\ti}^{\tio} (s-v)^{2H-1} dv ds\right)^2\right)^2 \\
&\hspace{-4em}\underset{N\to+\infty}{\longrightarrow} 0.
\end{align*}
We now turn to Term $C_3(\ti,\tio)$, for which we have :
\begin{align*}
&\E\left[\left|\sum_{i=0}^{N-1} C_3(\ti,\tio)\right|^2\right]\\
&\hspace{-4em}\leq C \sum_{i,i'=0}^{N-1} \int_{t_{i'+1}\vee t_{i+1}}^T \int_u^T \int_u^T \int_{t_i}^{t_{i+1}} \int_{t_i}^{t_{i+1}} (s-r)^{2H-1} (s'-r')^{2H-1} \E\left[|\delta_{\ell,i,s}(W^{2,H}) \delta_{\ell,i',s'}(W^{2,H})|\right] dr dr' ds ds' du\\
&\hspace{-4em}\leq C \left(\sum_{i=0}^{N-1} \left(\int_{\tio}^T \int_{t_i}^{t_{i+1}} (s-r)^{2H-1} \left(\int_{t_i}^{\tio} (s-v)^{2H-1} dv\right)^{1/2} dr ds\right)^2\right)^2\\
&\hspace{-4em}\leq C \left(\sum_{i=0}^{N-1} \left(\int_{\tio}^T \int_{t_i}^{\tio} (s-v)^{2H-1} dv\right)^{3} ds\right)^2\\
&\hspace{-4em}\underset{N\to+\infty}{\longrightarrow} 0.
\end{align*}
Following the same lines and using once again the uniform boundedness of derivatives (spatial and in the Malliavin sense) of $f$, we get immediately that 
\begin{align*}
\E\left[\left|\sum_{i=0}^{N-1} C_4(\ti,\tio)\right|^2\right]&\leq C \left(\sum_{i=0}^{N-1} \left(\int_{\tio}^T \int_{t_i}^{\tio} |t_{i+1}-t_i|^{\min\{2H,1\}} dv ds\right)^2\right)^2\\
&\underset{N\to+\infty}{\longrightarrow} 0.
\end{align*}

\textbf{Proof of (iii)\\\\}
We have $I_{1,7}(\ti,\tio)=\frac12 \sum_{k,\ell=1;k\neq\ell}^d I_{1,7,k,\ell}(\ti,\tio)$ with 
$$ I_{1,7,k,\ell}(\ti,\tio):=\int_{t_{i+1}}^T \frac{\partial^2}{\partial x_k \partial x_\ell} P_{\frac{1}{2H}(s-t_i)^{2H}} f^a(s,t_i,W^{2,H}(t_i,s)) \; \delta_{k,i,s}(W^{2,H}) \delta_{\ell,i,s}(W^{2,H}) ds. $$
Fix $k\neq \ell$ and set : 
$$ \rho_{i,k,\ell,s} :=\frac{\partial^2}{\partial x_k \partial x_\ell} P_{\frac{1}{2H}(s-t_i)^{2H}} f^a(s,t_i,W^{2,H}(t_i,s)).$$
We have 
\begin{align*}
&\E\left[\left|\sum_{i=0}^{N-1} I_{1,7,k,\ell}(\ti,\tio)\right|^2\right]\\
&=\E\left[\left|\sum_{i=0}^{N-1} \int_{t_{i+1}}^T \frac{\partial^2}{\partial x_k \partial x_\ell} P_{\frac{1}{2H}(s-t_i)^{2H}} f^a(s,t_i,W^{2,H}(t_i,s)) \; \delta_{k,i,s}(W^{2,H}) \delta_{\ell,i,s}(W^{2,H}) ds\right|^2\right]\\
&=2 \sum_{i,i'=0;i<i'}^{N-1} \int_{t_{i+1}}^T \int_{t_{i'+1}}^T \E\left[ \rho_{i,k,\ell,s} \rho_{i',k,\ell,s'} \delta_{k,i,s}(W^{2,H}) \delta_{\ell,i,s}(W^{2,H}) \underbrace{\E_{t_{i'}}\left[\delta_{k,i',s'}(W^{2,H}) \delta_{\ell,i',s'}(W^{2,H})\right]}_{=0}\right] ds ds'\\
&+\sum_{i=0}^{N-1} \int_{t_{i+1}}^T \int_{t_{i+1}}^T \E\left[ \rho_{i,k,\ell,s} \rho_{i,k,\ell,s'} \delta_{k,i,s}(W^{2,H}) \delta_{\ell,i,s}(W^{2,H}) \delta_{k,i',s'}(W^{2,H}) \delta_{\ell,i',s'}(W^{2,H})\right] ds ds'\\
&\leq C \sum_{i=0}^{N-1} \left(\int_{t_{i+1}}^T \int_{\ti}^{\tio} (s-v)^{2H-1} dv ds\right)^2 \\
&\leq C N^{-1} \sum_{i=0}^{N-1} \int_{t_{i+1}}^T \int_{\ti}^{\tio} (s-v)^{4H-2} dv ds \\
&= C N^{-1} \sum_{i=0}^{N-1} \left[(T-t_i)^{4H}-(t_{i+1}-t_i)^{4H}-(T-t_{i+1})^{4H}\right] \\
&\underset{N\to+\infty}{\longrightarrow} 0.
\end{align*}
\textbf{Proof of (iv)\\\\}
Term $I_{1,8}(\ti,\tio)$

$$ I_{1,8}(\ti,\tio):=\sum_{j=1}^d \sum_{k,\ell=1;k\neq \ell}^d I_{1,8,k,\ell}(\ti,\tio),$$
with
$$ I_{1,8,j,k,\ell}(\ti,\tio):= \frac12 \int_{t_{i+1}}^T \int_{t_{i}}^{\tio}\frac{\partial^2}{\partial x_k \partial x_\ell} P_{\frac{1}{2H}(s-t_i)^{2H}} g_j(s,u,W^{2,H}(t_i,s)) dB_j(u) \; \delta_{k,i,s}(W^{2,H}) \delta_{\ell,i,s}(W^{2,H}) ds.$$
Fix $k\neq \ell$, $j$. Set 
$$ \gamma_{k,\ell,i,u,s}:=\frac{\partial^2}{\partial x_k \partial x_\ell} P_{\frac{1}{2H}(s-t_i)^{2H}} g_j(s,u,W^{2,H}(t_i,s)).$$
Fix $i$, we have
$$ \int_{t_{i+1}}^T \int_{t_{i}}^{\tio} \gamma_{k,\ell,i,u,s} dB_j(u) \; \delta_{k,i,s}(W^{2,H}) \delta_{\ell,i,s}(W^{2,H}) ds = \int_{t_{i}}^{\tio} \int_{t_{i+1}}^T \gamma_{k,\ell,i,u,s} \delta_{k,i,s}(W^{2,H}) \delta_{\ell,i,s}(W^{2,H}) ds dB_j(u).$$
Then, it follows that
\begin{align*}
&\E\left[\left|\sum_{i=0}^{N-1} I_{1,8,j,k,\ell}(\ti,\tio)\right|^2\right]\\
&= \E\left[\left|\sum_{i=0}^{N-1} \int_{t_{i}}^{\tio} \int_{t_{i+1}}^T \gamma_{k,\ell,i,u,s} \delta_{k,i,s}(W^{2,H}) \delta_{\ell,i,s}(W^{2,H}) ds dB_j(u) \right|^2\right]\\
&= \sum_{i=0}^{N-1} \int_{t_{i}}^{\tio} \E\left[ \left| \int_{t_{i+1}}^T \gamma_{k,\ell,i,u,s} \delta_{k,i,s}(W^{2,H}) \delta_{\ell,i,s}(W^{2,H}) ds \right|^2 \right] du\\
&= \sum_{i=0}^{N-1} \int_{t_{i+1}}^T \int_{t_{i+1}}^T \E\left[ \int_{t_{i}}^{\tio} \gamma_{k,\ell,i,u,s} \gamma_{k,\ell,i,u,s'} du \; \delta_{k,i,s}(W^{2,H}) \delta_{k,i,s'}(W^{2,H}) \delta_{\ell,i,s}(W^{2,H}) \delta_{\ell,i,s'}(W^{2,H}) \right] ds ds' \\
&\leq C N^{-1} \sum_{i=0}^{N-1} \left(\int_{t_{i+1}}^T \int_{\ti}^{\tio} (s-v)^{2H-1} dv ds\right)^2 \\
&\underset{N\to+\infty}{\longrightarrow} 0.
\end{align*}
\textbf{Proof of (v)\\\\}

Term $I_{1,9}(\ti,\tio)$

$$ I_{1,9}(\ti,\tio):=\frac12 \sum_{j=1}^d \sum_{k,\ell=1;k\neq \ell}^d I_{1,9,k,\ell}(\ti,\tio),$$
with
$$\hspace{-2em} I_{1,9,j,k,\ell}(\ti,\tio):= \frac12 \int_{t_{i+1}}^T \int_{t_{i+1}}^T \int_{t_{i+1}}^{s}\frac{\partial^2}{\partial x_k \partial x_\ell} P_{\frac{1}{2H}(s-t_i)^{2H}} g_j(s,u,W^{2,H}(t_i,s)) dB_j(u) \; \delta_{k,i,s}(W^{2,H}) \delta_{\ell,i,s}(W^{2,H}) ds.$$
Fix $k\neq \ell$, $j$. Set 
$$ \gamma_{k,\ell,i,u,s}:=\frac{\partial^2}{\partial x_k \partial x_\ell} P_{\frac{1}{2H}(s-t_i)^{2H}} g_j(s,u,W^{2,H}(t_i,s)).$$


\begin{align*}
&\E\left[\left|\sum_{i=0}^{N-1} I_{1,9,j,k,\ell}(\ti,\tio)\right|^2\right]\\
&=\E\left[\left|\sum_{i=0}^{N-1} \int_{t_{i+1}}^T \int_{t_{i+1}}^{s}\gamma_{k,\ell,i,u,s} dB_j(u) \; \delta_{k,i,s}(W^{2,H}) \delta_{\ell,i,s}(W^{2,H}) ds \right|^2\right]\\
&=2 \sum_{i,i'=0;i<i'}^{N-1} \int_{t_{i+1}}^T \int_{t_{i'+1}}^T \E\left[\int_{t_{i+1}}^{s}\gamma_{k,\ell,i,u,s} dB_j(u) \; \delta_{k,i,s}(W^{2,H}) \delta_{\ell,i,s}(W^{2,H}) \delta_{k,i',s'}(W^{2,H}) \delta_{\ell,i',s'}(W^{2,H}) \times \right.\\ &\hspace{12em}\left.\underbrace{\E_{t_{i'+1}}\left[\int_{t_{i'+1}}^{s'}\gamma_{k,\ell,i',u',s'} dB_j(u')\right]}_{=0}\right]ds ds'\\
&=\sum_{i=0}^{N-1} \int_{t_{i+1}}^T \int_{t_{i+1}}^T \E\left[\int_{t_{i+1}}^{s\wedge s'} \gamma_{k,\ell,i,u,s} \gamma_{k,\ell,i,u,s'} du \; \delta_{k,i,s}(W^{2,H}) \delta_{\ell,i,s}(W^{2,H}) \delta_{k,i,s'}(W^{2,H}) \delta_{\ell,i,s'}(W^{2,H}) \right]ds ds' \\
&\leq C \sum_{i=0}^{N-1} \left(\int_{t_{i+1}}^T \int_{t_i}^{\tio} (s-v)^{2H-1} dv ds\right)^2 \\
&\underset{N\to+\infty}{\longrightarrow} 0.
\end{align*}

\textbf{Proof of (vi)\\\\}
Term $I_{1,10}(\ti,\tio)$

$$ I_{1,10}(\ti,\tio):=\sum_{j,k,\ell=1}^d I_{1,10,j,k,\ell}(\ti,\tio),$$
with 
$$ I_{1,10,j,k,\ell}(\ti,\tio):= \frac16 \int_{\tio}^T \mu_{i,s,j,k,\ell} \; \delta_{j,i,s}(W^{2,H}) \delta_{k,i,s}(W^{2,H}) \delta_{\ell,i,s}(W^{2,H}) ds,$$
where 
$$ \mu_{i,s,j,k,\ell}:=\int_0^1 \frac{\partial^3}{\partial x_j \partial x_k \partial x_\ell} P_{\frac{1}{2H}(s-t_i)^{2H}} f\left(s,W^{2,H}(t_i,s,\theta)\right) d\theta $$

\begin{align*}
&\E\left[\left|\sum_{i=0}^{N-1} I_{1,10,j,k,\ell}(\ti,\tio) \right|^2\right]^{1/2}\\
&\leq \sum_{i=0}^{N-1} \int_{\tio}^T \E\left[\left|\delta_{j,i,s}(W^{2,H}) \delta_{k,i,s}(W^{2,H}) \delta_{\ell,i,s}(W^{2,H})\right|^2\right]^{1/2} ds\\
&\leq C \sum_{i=0}^{N-1} \int_{\tio}^T \left(\int_{\ti}^{\tio} (s-v)^{2H-1} dv\right)^{3/2} ds\\ 
&\leq C N^{-1/2} \sum_{i=0}^{N-1} \int_{\tio}^T \int_{\ti}^{\tio} (s-v)^{3H-3/2} dv ds\\
&= C \left(N^{-1/2} \sum_{i=0}^{N-1} \left[(T-t_i)^{3H+1/2}-(T-\tio)^{3H+1/2}\right] \right)- C N^{-3H}\\
&\underset{N\to+\infty}{\longrightarrow} 0. 
\end{align*}
\end{proof}

\begin{lemma}
\label{lemma:jointmeas}
Let $f$ a smooth random field (that is $f \in \mathcal{S}_{ad}$). Then each term in this relation (\ref{eq:mainformula}) admits a version which jointly measurable in $(s,t,x,\omega)$ in $[0,T]^2\times\real^d\times \Omega$ ($s\leq t$). We will always consider this version.
\end{lemma}
\begin{proof}
Recall that $f$ (together with all its derivatives) is by definition bounded. The result is true for all the integrals in $dt$ as a consequence of Lebegue's dominated convergence. Concerning the terms involving a stochastic integral, we refer to \cite[Theorem IV.63]{Protter_V2}.
\end{proof}

\bibliographystyle{plain}
\bibliography{CDR.bib}

\end{document}